\documentclass[reqno,11pt]{amsart}
\usepackage{amsfonts,amsmath,amsthm,amssymb,stmaryrd,cite}
\usepackage{hyperref}
\usepackage{color}
\usepackage{mathrsfs}
\usepackage{esint}
\usepackage [latin1]{inputenc}
\usepackage{booktabs}
\usepackage{float}
\usepackage{graphics}
\usepackage{caption}
\usepackage{tikz}
\usepackage[export]{adjustbox}
\usepackage{subcaption}  
\usepackage{graphicx}   
\usepackage{float}  

\newtheorem{Theorem}{Theorem}[section]
\newtheorem{Lemma}[Theorem]{Lemma}

\newtheorem{Proposition}{Proposition}[section]
\newtheorem{Remark}{Remark}[section]
\newtheorem{Definition}{Definition}[section]
\numberwithin{equation}{section} \allowdisplaybreaks
\allowdisplaybreaks 
\date{\today}
\begin{document}
\author{Yi Peng}
\address{School of Sciences, Southwest Petroleum University, Chengdu 610500 China.}
\email{pengyi2573@126.com}

\author{Huaqiao Wang}
\address{College of Mathematics and Statistics, Chongqing University, Chongqing 401331 China.}
\email{wanghuaqiao@cqu.edu.cn}

\author{Chenlu Zhang}
\address{College of Mathematics and Statistics, Chongqing University, Chongqing 401331 China.}
\email{20220601002@stu.cqu.edu.cn}

\title[Non-uniqueness of the Hall-MHD equations]
{Non-uniqueness of weak solutions to the 3D Hall-MHD equations on the plane}
\thanks{Corresponding author: wanghuaqiao@cqu.edu.cn}
\keywords{The 3D Hall-MHD equations on the plane, weak solutions, non-uniqueness, convex integration.}
\subjclass[2010]{35Q35; 76W05; 35D30.}

\begin{abstract}
We prove the non-uniqueness of weak solutions with non-trivial magnetic fields to the 3D Hall-MHD equations on the plane in the space $C^0_t L_x^2$ through the convex integration scheme and by constructing new errors and new intermittent flows. In particular, based on the construction of 3D intermittent flows, we obtain the $2\frac{1}{2}$D Mikado flows through a projection onto the plane. Moreover, we prove that the constructed weak solution do not conserve the magnetic helicity and find that weak solutions of the ideal Hall-MHD equations in $C^{\bar{\beta}}_{t,x}$ ($\bar{\beta}>0$) are the strong vanishing viscosity and resistive limit of weak solutions to the Hall-MHD equations.
\end{abstract}

\maketitle
\section{Introduction and main result}
\subsection{Introduction}
The hypo viscous and resistive Hall magnetohydrodynamics (Hall-MHD) equations have the following form:
\begin{equation}\label{Hall-mhd}
\begin{cases}
\partial _tu+\nu_1(-\Delta)^{\alpha_1} u+u\cdot\nabla u+\nabla P= (\nabla\times B)\times B,&{\rm div}u = 0, \\
\partial _tB+\nu_2(-\Delta)^{\alpha_2} B-\nabla\times (u\times B)+\eta\nabla\times\left( (\nabla\times B)\times B\right)=0,&{\rm div}B=0,\\
(u,B)|_{t=0}=(u_0,B_0),
\end{cases}
\end{equation}
where $\alpha_1$, $\alpha_2$ are non-negative constants, $u,B$ and $P$ represent the velocity field, magnetic field and pressure of the fluid, respectively. The positive constants $\nu_1$ and $\nu_2$ represent the viscosity and resistivity coefficient, respectively. While the constant $\eta>0$ measures the magnitude of the Hall effect compared to the typical length scale of the fluid. The Hall-MHD equations has been studied for a long time in physics. The Hall term is key to solve the magnetic field reconnection in space plasmas \cite{f,hg}, star formation \cite{bt,w} and geo-dynamo \cite{mgm}, etc. Also, in the field of mathematics, the Hall-MHD equations attract a lot of attention and there is a large amount of results, see \cite{adfl,cdl,CL,CS,CWW,CW1,CW,DM,DM1,DH1,dt,DH} and the references therein, for example.

For the non-uniqueness of weak solutions to the Hall-MHD equations, there is a small amount of researches. Dai \cite{DM1} proved the non-uniqueness of Leray--Hopf weak solutions to the standard 3D Hall-MHD equation. Concerning the 3D generalized Hall-MHD equations with the viscosity and resistivity exponent $\alpha\in [\rho,5/4)$, where $\rho\in(1,5/4)$, the authors \cite{pw} obtained the non-uniqueness of a class of weak solutions, which are smooth in time outside a set whose Hausdorff dimension strictly less than 1.

If $\eta=0$, \eqref{Hall-mhd} will be the MHD system:
\begin{equation}\label{mhd}
\begin{cases}
\partial _tu+\nu_1(-\Delta)^{\alpha_1} u+u\cdot\nabla u+\nabla P= (\nabla\times B)\times B,&{\rm div}u = 0, \\
\partial _tB+\nu_2(-\Delta)^{\alpha_2} B-\nabla\times (u\times B)=0,&{\rm div}B=0,\\
(u,B)|_{t=0}=(u_0,B_0).
\end{cases}
\end{equation}
In 3D space, there are some interesting results about the non-uniqueness of weak solutions to the MHD equations. For the case of $\alpha_1=\alpha_2=\alpha\in(0,5/4)$, Li-Zeng-Zhang \cite{lzz} established the non-uniqueness of weak solutions by using convex integration. Later, they obtained non-uniqueness of weak solutions in $L_t^\gamma L_x^\infty$, $1\leq\gamma<2$, which is sharp in view of the classical Ladyzhenskaya-Prodi-Serrin criteria, and constructed weak solutions to the hyper viscous and resistive MHD beyond the Lions exponent, in the space $L^\gamma_tW^{s,p}_x$, where the exponents $(s,\gamma,p)$ lie in two supercritical regimes \cite{lzz1}. Furthermore, they also provided the observation that the Taylor's conjecture fail in the limit of a sequence of non-Leray--Hopf weak solutions for the hyperdissipative MHD equations, which contrasted to the positive result of weak ideal limit \cite{fl}, namely, the weak limit of Leray--Hopf solutions to the MHD system. Recently, Beekie-Buckmaster-Vicol \cite{bbv} constructed weak solutions to the ideal MHD equations based on a Nash-type convex integration scheme with intermittent building blocks. Moreover, they also showed that the magnetic helicity of those kind of weak solutions is not a constant, which implies that Taylor's conjecture is false. For more results about non-uniqueness of the 3D MHD system, see \cite{df,my,ny} and the references therein.

While the magnetic field is ignored, \eqref{Hall-mhd} reduces to the Navier--Stokes equation (NSE):
\begin{equation}\label{ns}
\begin{cases}
\partial _tu+\nu_1(-\Delta)^{\alpha_1} u+{\rm div}(u\otimes u)+\nabla p=0,{\rm div}u = 0, \\
u|_{t=0}=u_0.
\end{cases}
\end{equation}
Leray \cite{leray} and Hopf \cite{h} constructed the weak solutions to the standard 3D NSE, that is $\alpha_1=1$, which is now referred to as Leray--Hopf weak solutions. However, the uniqueness still remains open. In the breakthrough work, Buckmaster-Vicol \cite{bv} proved that the weak solution of the 3D NSE is not unique in $C_tL^2_x$ by using intermittent convex integration. Later, Albritton-Br\'{u}e-Colombo \cite{ABC} showed non-uniqueness of Leray--Hpof weak solutions of the NSE with a special force by constructing a background solution. While for the stationary NSE, Luo and Cheskidov \cite{CL, Luo} proved the non-uniqueness of weak solutions.

For the hyper-viscous NSE, J.-L. Lions \cite{l} proved that for any $d\geq2$ dimensional NSE \eqref{ns}, the Leray--Hopf solution is unique when $\alpha_1\geq\frac{d+2}{4}$. In the three-dimensional space, the Lions exponent $\alpha_1=5/4$ is an important criterion for the well-posedness or the ill-posedness to the fractional NSE \eqref{ns}. If $\alpha_1$ is slightly less than $5/4$, Tao \cite{t} obtained the global well-posedness for \eqref{ns} with logarithmically supercritical dissipation. For the super-critical regime $\alpha_1<1/5$ and $\alpha_1<1/3$, the non-uniqueness of Leray solutions to \eqref{ns} is proved by \cite{cdd} and \cite{d}, respectively. For the case of $1<\alpha_1<5/4$, Luo-Titi \cite{lt1} obtained the non-uniqueness of the very weak solution for the system \eqref{ns}.  For $1\leq\alpha_1<5/4$, Buckmaster-Colombo-Vicol \cite{bcv} constructed the non-Leray--Hopf weak solutions in the super-critical spaces $C^0_tL^2_x$, and they also proved such weak solutions are smooth outside a fractal set of singular times with Hausdorff dimension strictly less than 1. For the 3D hyper-viscous NSE \eqref{ns} when the viscosity $\alpha_1$ is beyond the Lions exponent $5/4$, in the work of Li-Qu-Zeng-Zhang \cite{lqzz}, the non-uniqueness was proved in the supercritical spaces $L^\gamma_tW^{s,p}_{x}$ and also yielded that the sharpness at two endpoints of the Ladyzhenskaya-Prodi-Serrin criteria condition.

For the standard 2D NSE, that is, $\alpha_1=1$ in \eqref{ns}, there exists a unique Leray--Hopf weak solution, see \cite{temam,leray}, for example. However, the weaker solution may not unique, even if $\alpha_1\geq1$. Recently, Cheskidov-Luo \cite{cl} proved the non-uniqueness of weak solutions to the standard NSE in the super-critical space $L^\gamma L^\infty$ with $1\leq\gamma<2$, which is sharp in view of the classical Ladyzhenskaya-Prodi-Serrin criteria. Later, Luo-Qu \cite{lq} proved the non-uniqueness of weak solutions in $C_t^0L^2_x$ to the 2D hyperdissipative NSE \eqref{ns} for $0\leq\alpha_1<1$.  Cheskidov-Luo \cite{cl1} proved the sharp non-uniqueness in $C_tL^p_x$, $1\leq p<2$. Very recently, for $\alpha_1\in [1,\frac{3}{2})$, Du-Li \cite{dl} and Li-Tan \cite{lt} proved the non-uniqueness of weak solutions to \eqref{ns} in the spaces  $L^\gamma_tW^{s,p}_x$ (where the exponents $(s,\gamma,p)$ lie in two supercritical regimes) and $L^\gamma L^p$ ($\frac{4\alpha_1-4}{\gamma}+\frac{2}{p}>2\alpha_1-1$), respectively, which are sharp in view of the generalized Ladyzhenskaya-Prodi-Serrin criteria.

The existence and uniqueness of the standard 2D MHD system is obtained by Sermange-Temam \cite{st} for $\nu_1>0$, $\nu_2>0$ in \eqref{mhd}. While for the non-uniqueness of weak solution to the 2D ideal MHD system, that is $\nu_1=\nu_2=0$, symmetry reduction techniques are not applicable for the magnetic stress error. Miao-Nie-Ye\cite{mny} obtained the non-uniqueness result by splitting the velocity field, such that the spatial decoupling reduces the relaxed system to one that only contains Reynolds stress error.

However, for Hall-MHD equations, the Hall term decouples from the rest of the system in the 2D case. For this reason, physicists such as Donato et al. in \cite{dsdccm} have turned to the $2\frac{1}{2}$D case in which a pair of
\begin{align}\label{1.2}
u(x,t)=(u_1,u_2,u_3)(x_1,x_2,t)\quad{\rm and} \quad B(x,t)=(B_1,B_2,B_3)(x_1,x_2,t)
\end{align}
solves \eqref{Hall-mhd}.
In this paper, we first consider the non-uniqueness of weak solutions to the $2\frac{1}{2}$D flows as in \cite{mb} for Hall-MHD \eqref{Hall-mhd} on $\mathbb{T}^2\times[0,T]$, where $\alpha_1\in(0, \frac{3}{4})$, $\alpha_2\in(0, \frac{5}{4})$. The operator $\nabla$ is given by $\nabla=(\partial_{x_1},\partial_{x_2},0)^T$ so that
$${\rm div}u=\nabla\cdot u:=\partial_{x_1}u_1+\partial_{x_2}u_2,\quad\quad \nabla P=(\partial_{x_1}P,\partial_{x_2}P,0)^T,$$
and
$${\rm curl}B=\nabla\times B:=(\partial_{x_2}B_3,-\partial_{x_1}B_3,\partial_{x_1}B_2-\partial_{x_2}B_1)^T.$$

The $2\frac{1}{2}$D Hall-MHD systems are different from NS or MHD. In the 2D space, the existence and uniqueness of the Leray--Hopf weak solution to the standard NS  and MHD equations are obtained by Leray \cite{leray} and Sermange-Temam \cite{st} for $\nu_1>0$, $\nu_2>0$, and such results can be extended to the $2\frac{1}{2}$D case as well. Indeed, for these equations, planar component equations and the equation for $u_3$ or $B_3$ are decoupled. But for the $2\frac{1}{2}$D Hall-MHD systems, planar part and the third component are related by the Hall term. For such flow, Chae-Wolf\cite{cw} have proved the existence of Leray--Hopf weak solution, while the uniqueness is an open problem, see \cite{dt1,y}, for example. In the present paper, we prove the non-uniqueness of the weak solutions defined in Definition \ref{def1} to the $2\frac{1}{2}$D Hall-MHD equations \eqref{Hall-mhd}.

In analogy with the Onsager conjecture of the Euler systems, we also concern with the magnetic helicity:
\begin{align}\label{helicity}
\mathcal{H}_{B,B}(t)=\int_{\mathbb{T}^2} A(t,x)B(t,x)dx,
\end{align}
where $A$ is the mean-free periodic vector field satisfying ${\rm curl}A=B$. And for the $2\frac{1}{2}$D Hall-MHD systems, the magnetic helicity also defined as \eqref{helicity}, see \cite{lhf}. To the best knowledge of authors, there is no result relating to Onsager type conjecture for the magnetic helicity of the $2\frac{1}{2}$D Hall-MHD systems. While for the total energy
$$\mathcal{E}(t)=\frac{1}{2}\int |u(t,x)|^2+|B(t,x)|^2dx$$
of the ideal Hall-MHD equations, there are only two positive results \cite{kdb,wyy}. For the two-dimensional ideal MHD system, the magnetic helicity is replaced by the mean-square magnetic potential $\int_{\mathbb{T}^2}|\psi|^2dx$, where $\psi$ is the magnetic potential such that $B=\nabla^{\bot}\psi$. Furthermore, this quantity is conserved for weak solutions in $C_{w}L^2$, see \cite{fls}. However, for the Hall-MHD systems defined on the plane, the similar positive result may fail to the weak solutions, even in a smaller space. More precisely, we construct infinitely many weak solutions to the $2\frac{1}{2}$D Hall-MHD equation in the space $C_t^0L^2_x$, which do not conserve the magnetic helicity.

In the end of this subsection, we list notations which will be used in the following text.
\textbf{Notations:}
\begin{enumerate}
  \item For any positive $A$ and $B$, we use the notation $A\lesssim B$ to mean that there exists a positive constant $C$ such that $A\leqslant CB$.
  \item For every $p\in[1,\infty]$ and $s\in \mathbb{R}$, we denote the following short notations by
  $$L_x^{p}:=L_x^{p}(\mathbb{T}^2),\quad W_x^{s,p}:=W_x^{s,p}(\mathbb{T}^2).$$
  \item For any normed space $X$, we employ the notation $L^{p}\left([0,T],X\right)$ to denote the space of functions $f$ such that for almost all $t\in (0,T)$, $f(t)\in X$ and $\left\|f(t)\right\|_{X}\in L^{p}(0,T)$. 
      In particular, we write $L^p_{t,x}:=L^p_tL^p_x$ and $C_{t,x}:=C_tC_x$.
  \item The homogeneous Sobolev space $\dot{H}^{s}(\mathbb{R}^{3})$, for any $s\in \mathbb{R}$, as the subspace of tempered distributions whose Fourier transform is locally integrable and the following norm is finite:
      $$\left\|f\right\|_{\dot{H}^{s}}=\left(\int_{\mathbb{R}^{3}}\left|\xi\right|^{2s}\left|\hat{f}(\xi)\right|^{2}d\xi\right)^{\frac{1}{2}}.$$
  \item For any set $A\subseteq[0,T]$, set
  $$N_{\varepsilon_*}\left(A\right):=\{t\in[0,T]: \exists\; s\in A \;{\rm s.t}\; |t-s|\leq\varepsilon_*\}.$$
  \item We use the following short notations
  $${\rm supp}_t(u,B):={\rm supp}_t u\cup {\rm supp}_t B,\quad \|(u,B)\|_{X}:=\|u\|_{X}+\|B\|_{X},$$
  where $X$ is Banach space.
  \item For the 3D vector $u$, the traceless part of $u\otimes u$ is
  $$u\mathring{\otimes}u=u\otimes u-\frac{1}{3}|u|^2{\rm Id}.$$
  \item The projection onto the functions with zero mean $\mathbb{P}_{\neq0}f=f-\fint_{\mathbb{T}^2}$, and the Helmholtz projection $\mathbb{P}_{H}f=f-\nabla(\Delta^{-1}{\rm div}f)$.
\end{enumerate}


\subsection{Main result}\label{sec1.2}

To begin with, let us formulate precisely the definition of weak solutions to the Hall-MHD equations. Using the identities $(\nabla\times B)\times B={\rm div}\left( B\otimes B\right)$ and $\nabla\times (u\times B)={\rm div}(u\otimes B -B\otimes u)$, equation \eqref{Hall-mhd} can be written as:
\begin{equation}\label{Hall-MHD1}
\begin{cases}
\partial _tu+\nu_1(-\Delta)^{\alpha_1} u+{\rm div}(u\otimes u-B\otimes B)+\nabla p= 0,&{\rm div}u = 0, \\
\partial _tB+\nu_2(-\Delta)^{\alpha_2} B+{\rm div}(B\otimes u -u\otimes B)+\nabla\times{\rm div}\left( B\otimes B\right)=0,&{\rm div}B=0,\\
(u,B)|_{t=0}=(u_0,B_0),
\end{cases}
\end{equation}
where $p=P+\frac{1}{2}|B|^2$.
We focus on the following non-Leray--Hopf weak solution:
\begin{Definition}[Weak solution]\label{def1} Given any zero mean initial data $(u_0,B_0)\in L^2$, we say that $(u,B)\in C^0([0,T];L^2(\mathbb{T}^2))$ is a weak solution of the Cauchy problem for the Hall-MHD equations \eqref{Hall-MHD1} if $(u,B)(\cdot,t)$ is weakly divergence free for all $t\in[0,T]$, with zero mean, and
\begin{align*}
\int_{\mathbb{T}^2}u_0\varphi(\cdot,0)dx+\int_{0}^T\int_{\mathbb{T}^2}u\partial _t\varphi-\nu_1 u(-\Delta)^{\alpha_1} \varphi+(u\otimes u-B\otimes B):\nabla \varphi dxdt=0,
\end{align*}
\begin{align*}
\int_{\mathbb{T}^2}B_0\varphi(\cdot,0)dx+\int_{0}^T\int_{\mathbb{T}^2}&B\partial _t\varphi-\nu_2B(-\Delta)^{\alpha_2} \varphi+(u\otimes B-B\otimes u):\nabla \varphi \\
&+(B\otimes B):\nabla \nabla\times\varphi dxdt=0,
\end{align*}
for any $\varphi\in C_0^\infty(\mathbb{T}^2\times[0,T))$ such that $\varphi(\cdot,t)$ is divergence free for all $t$.
\end{Definition}

The main theorem of this paper is stated as follows, from which we can get the non-uniqueness of weak solutions of the Hall-MHD equation \eqref{Hall-MHD1}.
\begin{Theorem}\label{main result}
Let $\alpha_1\in(0, \frac{3}{4})$, $\alpha_2\in(0,\frac{5}{4})$, and let $(\tilde{u},\tilde{B})$ be any smooth divergence free and mean-value vector fields on $\mathbb{T}^2\times[0,T]$. There exists $\beta'\in (0,1)$ such that for any given $\varepsilon_*>0$, there exists a weak solution $(u,B)$ of system \eqref{Hall-MHD1} satisfying
\begin{align}\label{thm1.2-1}
(u,B)\in C^0([0,T];H^{\beta'}(\mathbb{T}^2)),
\end{align}
\begin{align}\label{thm1.2-2}
\|u-\tilde{u}\|_{C^0([0,T];L^{1}(\mathbb{T}^2))}\leq\varepsilon_*,\quad\|B-\tilde{B}\|_{C^0([0,T];L^{1}(\mathbb{T}^2))}\leq\varepsilon_*,
\end{align}
\begin{align}\label{thm1.2-3}
{\rm suppt}_t(u,B)\subseteq N_{\varepsilon_*}\left({\rm suppt}_t(\tilde{u},\tilde{B})\right),
\end{align}
and
\begin{align}\label{thm1.2-4}
|\mathcal{H}_{B,B}-\mathcal{H}_{\tilde{B},\tilde{B}}|\leq\varepsilon_*.
\end{align}
\end{Theorem}

Now, we give the statement of the non-uniqueness result.
\begin{Theorem}[Main Theorem]\label{main theorem}
For $\alpha_1\in(0, \frac{3}{4})$, $\alpha_2\in(0, \frac{5}{4})$, there exist infinitely many weak solutions of \eqref{Hall-MHD1} with the same data at zero, which live in $H^{\beta'}$ for some $\beta'\in (0,1)$ and do not conserve the magnetic helicity.
\end{Theorem}

\begin{Remark}
For the 2D NS system \eqref{ns} with $\alpha_1\in [0,1)$, the $C^0_tL^2_x$ weak solutions are not unique (see \cite{lq}). As we mentioned above, the non-uniqueness of weak solutions to the $2\frac{1}{2}$D NS systems can be deduced from that of the 2D NS equations. When the magnetic field is zero, the Hall-MHD systems becomes the NS equations; then the non-unique result of the NS system immediately gives the same result of the Hall-MHD system. However, the weak solutions constructed in Theorem \ref{main theorem} can not be obtained by such way, since the magnetic helicity of weak solutions $(u,B)=(u,0)\in C^0_tL^2_x$ is identically zero.
\end{Remark}

Our next result is the vanishing viscosity and resistive limits of weak solutions to the $2\frac{1}{2}$D Hall-MHD equations.
\begin{Theorem}\label{thm1.4}
Let $\alpha_1\in(0, \frac{3}{4})$, $\alpha_2\in(0, \frac{5}{4})$ and $(u,B)\in C^{\bar{\beta}}_{t,x}\times C^{\bar{\beta}}_{t,x}$ be any mean-free solution to the ideal Hall-MHD equations, where $\bar{\beta}>0$. Then there exists $\beta'\in (0,\bar{\beta})$ and a sequence of weak solutions $(u^{\nu_{1,n}},B^{\nu_{2,n}})\in C^0_t H^{\beta'}_x\times C^0_t H^{\beta'}_x$ to the Hall-MHD system \eqref{Hall-MHD1} such that
$$u^{\nu_{1,n}}\rightarrow u,\quad B^{\nu_{2,n}}\rightarrow B \text{ strongly in}\;\; C^0_t L^2_x \text{ as } \nu_n:=(\nu_{1,n}, \nu_{2,n})\to 0,$$
where $\nu_{1,n}$, $\nu_{2,n}$ are the viscosity and resistivity coefficients, respectively.
\end{Theorem}

\subsection{Outline of the proof}
Without loss of generality, we take $T=1$ in the rest of the paper. The construction of the non-Leray--Hopf weak solution is based on the intermittent convex integration scheme, inspired by \cite{bv}. More precisely, for $q\in \mathbb{N}$, we focus on the following approximating system: {\small
\begin{equation}\label{Hall-MHDapp}
\begin{cases}
\partial _tu_q+\nu_1(-\Delta)^{\alpha_1} u_q+{\rm div}(u_q\otimes u_q-B_q\otimes B_q)+\nabla p_q={\rm div}\mathring{R}_q^u, \\
\partial _tB_q+\nu_2(-\Delta)^{\alpha_2} B_q+{\rm div}(B_q\otimes u_q-u_q\otimes B_q)+\nabla\times{\rm div}\left( B_q\otimes B_q\right)=\nabla\times{\rm div}\mathring{R}_q^B,\\
{\rm div}u_q ={\rm div}B_q=0,
\end{cases}
\end{equation}}
\!\!where the Reynolds stress $\mathring{R}_q^u$ and the magnetic stress $\mathring{R}_q^B$ are symmetric traceless $3\times3$ matrices. For $q\in \mathbb{N}$, define $\lambda_q$ and $\delta_{q+2}$ as follows:
\begin{align}{\label{1.9}}
\lambda_q=a^{b^q},\quad \delta_{q+2}=\lambda_{q+2}^{-2\beta}.
\end{align}
Here $a\in 5\mathbb{N}$ is a large integer such that $a^\varepsilon\in 5\mathbb{N}$, $b\in 2\mathbb{N}$ satisfying
\begin{align}{\label{1.10}}
b>\frac{1000}{\varepsilon},\quad 0<\beta<\frac{1}{100b^2},
\end{align}
where $\varepsilon$ is a small parameter with
\begin{align}{\label{1.11}}
0<\varepsilon\leq\frac{1}{4}{\rm min}\left\{\frac{1}{2},\frac{3}{4}-\alpha_1,\frac{5}{4}-\alpha_2\right\}.
\end{align}
Let $q\in \mathbb{N}$, we suppose the following inductive estimates of system \eqref{Hall-MHDapp} hold at level $q$:
\begin{align}{\label{1.12}}
\left\|(u_q,B_q)\right\|_{C^1_{t,x}}\leq \lambda_q^4,
\end{align}
\begin{align}{\label{1.13}}
\left\|(\mathring{R}_q^u,\mathring{R}_q^B)\right\|_{C^1_{t,x}}\leq \lambda_q^8,
\end{align}
\begin{align}{\label{1.14}}
\left\|(\mathring{R}_q^u,\mathring{R}_q^B)\right\|_{L^1_{x}}\leq \delta_{q+1}.
\end{align}
The weak solutions are obtained by a sequence of approximating solutions, which will be constructed by the following iteration theorem.
\begin{Theorem}[Main Iteration]\label{main iteration}
Let $\alpha_1\in(0, \frac{3}{4})$, $\alpha_2\in(0, \frac{5}{4})$. There exist $\beta\in (0,1)$, $M>0$ large enough and $a_0=a_0(\beta,M)$ such that for any integer $a\geq a_0$, the following holds:

Suppose that $\left(u_q,B_q,\mathring{R}_q^u,\mathring{R}_q^B\right)$ solves \eqref{Hall-MHDapp}and satisfies \eqref{1.12}-\eqref{1.14}. Then there exists a new pair $(u_{q+1},B_{q+1},\mathring{R}_{q+1}^u,\mathring{R}_{q+1}^B)$ solves \eqref{Hall-MHDapp} and satisfies \eqref{1.12}-\eqref{1.14} with $q+1$ replacing $q$. In addition,
\begin{align}{\label{1.15}}
\left\|(u_{q+1}-u_q,B_{q+1}-B_q)\right\|_{L^2_x}\leq M\delta_{q+1}^{\frac{1}{2}},
\end{align}
\begin{align}{\label{1.16}}
\left\|(u_{q+1}-u_q,B_{q+1}-B_q)\right\|_{L^{1}_x}\leq \delta_{q+2}^{\frac{1}{2}},
\end{align}
\begin{align}{\label{1.17}}
{\rm supp}_t\left(u_{q+1},B_{q+1},\mathring{R}_{q+1}^u,\mathring{R}_{q+1}^B\right)\subseteq N_{\delta_{q+2}^{\frac{1}{2}}}\left({\rm supp}_t\left(u_q,B_q,\mathring{R}_q^u,\mathring{R}_q^B\right)\right).
\end{align}
\end{Theorem}
The new pair $\left(u_{q+1},B_{q+1},\mathring{R}_{q+1}^u,\mathring{R}_{q+1}^B\right)$ will be constructed in Section \ref{sec2}--\ref{sec5}. And the main results listed in Section \ref{sec1.2} will be proved by Theorem \ref{main iteration} in Section \ref{sec6}.

\subsection{New ingredient of the proof}
In the frame of convex integration for $2\frac{1}{2}$D Hall-MHD equations, the construction of intermittent velocity and magnetic flows is different from the previous works. In other words, there is no such $2\frac{1}{2}$ dimensional intermittent flows in the existing literatures. Concerning the routine of convex integration in $\mathbb{R}^3$, $(k,A_k,k\times A_k)$ is an orthonormal bases of $\mathbb{R}^3$, and the intermittent flows have the following form:
$$W_{(k)}=\psi_{(k)}\phi_{(A_k,k\times A_k)}k,$$
where $\{\psi_{(k)}\phi_{(A_k,k\times A_k)}\}$ are spatial concentration functions with three oscillation directions $k$, $A_k$, $k\times A_k$, and the set of $k$, named $\Lambda\subseteq \mathbb{S}\cap \mathbb{Q}^3$, satisfying geometric lemma \ref{lem3.1}, see \cite{bcv,pw}, for instance. For the $2\frac{1}{2}$ dimensional flows, based on the construction of $\Lambda$ and $(k,A_k,k\times A_k)$, we define the orthonormal bases $(\tilde{k},\tilde{A_k})$ of $\mathbb{R}^2$ by a projection onto the plane, such that $k\cdot(\tilde{A_k},0)=0$, see Appendix \ref{sec7} for more details. If we choose the intermittent flow $W_{(k)}$ by a similar way in \cite{bcv}, that is,
$$W_{(k)}=\psi_{(\tilde{k})}\phi_{(\tilde{A_k})}k,\quad k\in \Lambda,$$
which has two oscillation directions $\tilde{k}$, $\tilde{A_k}$, we find that $(\tilde{k},0)\times k$ may not disappear while $k\times k=0$ in $\mathbb{R}^3$. Then it is hard to construct the divergence-free corrector of $W_{(k)}$. Thus we consider the Mikado flows $W_{(k)}=\phi_{(\tilde{A_k})}k$, which satisfies ${\rm div}W_{(k)}=0$.

Another key point is the estimation of the interaction oscillations, that is $\mathring{R}_{osc,far}^u$ and $\mathring{R}_{osc,far}^B$ in Section \ref{sec5.2}. For every $k\in \Lambda_1\cup\Lambda_2$, the support of $W_{(k)}$ is contained in a thin rectangle with width $\sim(\mu\sigma)^{-1}$, where $\mu$ and $\sigma$ are large parameters defined in \eqref{para}, while the length equals to the side length of the period of $W_{(k)}$, which implies that the shifts such that $\{W_{(k)}\}$ have disjoint supports may not exist. On the other hand, we have no idea if the intersections between different supports of $W_{(k)}$ and $W_{(k')}$ contained in much smaller rectangles like the Mikado flows in \cite{lzz}, since $\tilde{A_k}$ and $\tilde{A_{k'}}$ may coincide ($\tilde{k}=\tilde{k'}$ at the same time) even if $k\neq k'$, see Table \ref{tab3} in Appendix \ref{sec7}. Thus the intersection oscillation can not be handled by the same way of 2D or 3D Mikado flows in \cite[Section 4.1]{cl} and \cite[Lemma 3.4]{lzz}. To overcome the obstacle, we modify the method in \cite[Lemma 3.4]{lzz} to estimate the intersection oscillations. For the case of $k\neq k'$ and $\tilde{A_k}\neq\tilde{A_{k'}}$, the intersections of the supports of $W_{(k)}$ and $W_{(k')}$ must be contained in much more smaller rectangles with length and width $\sim(\mu\sigma)^{-1}$, see Figure \ref{fig1}. In this setting, the intersection oscillations can be handled by a similar way in \cite[Lemma 3.4]{lzz}. While for the case of $k\neq k'$ but $\tilde{A_k}=\tilde{A_{k'}}$, the supports of $W_{(k)}$ and $W_{(k')}$ are contained in parallel thin rectangles with width $\sim(\mu\sigma)^{-1}$, thus we can find shifts to make sure that such $W_{(k)}$ and $W_{(k')}$ have disjoint supports.


\section{Mollification}\label{sec2}
In order to avoid the loss of derivatives, we mollify the velocity and magnetic fields. Let $\phi_\epsilon$, $\varphi_\epsilon$ be the standard  2D and 1D Friedrichs mollifiers, and ${\rm supp}\varphi_\epsilon\subseteq (-\epsilon,\epsilon)$, for $\epsilon>0$. The mollifications of $\left(u_q,B_q,\mathring{R}_q^u,\mathring{R}_q^B\right)$ in space and time are given by
\begin{align}\label{2.1}
&u_l:=(u_q\ast_x\phi_l)\ast_t\varphi_l,\quad B_l:=(B_q\ast_x\phi_l)\ast_t\varphi_l,\notag \\
&\mathring{R}_l^u:=(\mathring{R}_q^u\ast_x\phi_l)\ast_t\varphi_l,\quad \mathring{R}_l^B:=(\mathring{R}_q^B\ast_x\phi_l)\ast_t\varphi_l,
\end{align}
where $l=\lambda_q^{-20}$. Then in view of system \eqref{Hall-MHDapp}, $\left(u_l,B_l,\mathring{R}_l^u,\mathring{R}_l^B\right)$ satisfies
\begin{equation}\label{Hall-MHDmollify}
\begin{cases}
\partial _tu_l+\nu_1(-\Delta)^{\alpha_1} u_l+{\rm div}(u_l\mathring{\otimes} u_l-B_l\mathring{\otimes} B_l)+\nabla p_l={\rm div}(\mathring{R}_l^u+\mathring{R}^u_{com}), \\
\partial _tB_l+\nu_2(-\Delta)^{\alpha_2} B_l+{\rm div}(B_l\otimes u_l-u_l\otimes B_l)+\nabla\times{\rm div}\left( B_l\mathring{\otimes}B_l\right)\\
\quad\quad\quad\quad\quad\quad\quad\quad\quad=\nabla\times{\rm div}(\mathring{R}_l^B+\mathring{R}^B_{com,1})+{\rm div}\mathring{R}^B_{com,2},\\
{\rm div}u_q ={\rm div}B_q=0,
\end{cases}
\end{equation}
where the traceless symmetric commutator stresses $\mathring{R}^u_{com}$ and $\mathring{R}^B_{com,1}$; the skew-symmetric commutator stress $\mathring{R}^B_{com,2}$ are given by
\begin{align}
&\mathring{R}^u_{com}:=u_l\mathring{\otimes} u_l-B_l\mathring{\otimes} B_l-(u_q\mathring{\otimes} u_q-B_q\mathring{\otimes} B_q)\ast\phi_l\ast\varphi_l,\label{2.3}\\
&\mathring{R}^B_{com,1}:=B_l\mathring{\otimes}B_l-(B_q\mathring{\otimes}B_q)\ast\phi_l\ast\varphi_l,\label{2.4}\\
&\mathring{R}^B_{com,2}:=B_l\otimes u_l-u_l\otimes B_l-(B_q\otimes u_q-u_q\otimes B_q)\ast\phi_l\ast\varphi_l,\label{2.5}
\end{align}
and the pressure $p_l$ is given by
\begin{align}\label{2.6}
p_l:=p_q\ast\phi_l\ast\varphi_l+\frac{1}{3}(|u_q|^2-|B_q|^2)\ast\phi_l\ast\varphi_l.
\end{align}
By standard millification estimates and inductive estimates \eqref{1.12}-\eqref{1.14}, for any $N\in \mathbb{N}_+$, and $1<p<\infty$, one has
\begin{align}
&\|(u_l,B_l)\|_{C^N_{t,x}}\lesssim l^{-N+1}\|(u_q,B_q)\|_{C^1_{t,x}}\lesssim l^{-N+1} \lambda_q^4\lesssim l^{-N},\label{2.7}\\
&\|(\mathring{R}^u_l,\mathring{R}^B_l)\|_{C^N_{t,x}}\lesssim l^{-N+1}\|(\mathring{R}^u_l,\mathring{R}^B_l)\|_{C^1_{t,x}}\lesssim l^{-N+1} \lambda_q^8\lesssim l^{-N},\label{2.8}\\
&\|(\mathring{R}^u_l,\mathring{R}^B_l)\|_{L^1_{x}}\lesssim \|(\mathring{R}^u_l,\mathring{R}^B_l)\|_{L^1_{x}}\lesssim \delta_{q+1},\label{2.9}\\
&\|\mathring{R}^u_{com}\|_{L^p_{x}}\lesssim \|\mathring{R}^u_{com}\|_{C_{t,x}}\lesssim l\|(u_q,B_q)\|^2_{C^1_{t,x}}\lesssim l \lambda_q^{8}\lesssim \lambda_q^{-12}.\label{2.10}
\end{align}
Similarly,
\begin{align}
&\|(\mathring{R}^B_{com,1},\mathring{R}^B_{com,2})\|_{L^p_{x}}\lesssim \lambda_q^{-12}.\label{2.11}
\end{align}
By the definitions of $\mathring{R}^u_l$ and $\mathring{R}^B_l$, one has
\begin{align}
&{\rm suppt}_t\mathring{R}^u_l\subseteq N_l({\rm suppt}_t\mathring{R}^u_q),\label{2.12}\\
&{\rm suppt}_t\mathring{R}^B_l\subseteq N_l({\rm suppt}_t\mathring{R}^B_q).\label{2.13}
\end{align}


\section{Intermittent flows}\label{sec3}
We set the scaling parameters $\mu$ and $\sigma$ as
\begin{align}\label{para}
\mu=\lambda_{q+1},\quad \sigma=\lambda_{q+1}^\varepsilon,
\end{align}
where $\varepsilon$ is the small constant in \eqref{1.11}.

Let's recall the following geometric lemma in \cite{bcv}.
\begin{Lemma}\label{lem3.1}
For $\alpha=1,2$, there exist disjoint subsets $\Lambda_\alpha\subset \mathbb{S}^2\cap\mathbb{Q}^3$ and smooth functions $\Gamma_k: \mathcal{N}\rightarrow\mathbb{R}$ such that
\begin{align}
R=\sum_{k\in \Lambda_\alpha}\Gamma_k^2(R)k\otimes k,
\end{align}
for symmetric matrix $R$ satisfying $|R-{\rm Id}|\leq\delta$, where $\delta$ is a small constant.
\end{Lemma}
Foe convenience, we give the proof of this lemma in Appendix \ref{sec7}. Take $N_\Lambda\in \mathbb{N}$ such that
\begin{align}
\left\{N_\Lambda k,N_\Lambda A_k,N_\Lambda k\times A_k\right\}\subset N_\Lambda\mathbb{S}^2\cap\mathbb{N}^3,
\end{align}
for every $k\in \Lambda_1\cup\Lambda_2$, where $(k,A_k,k\times A_k)$ is orthonormal basis of $\mathbb{R}^3$.

Let $\Phi\in C^\infty_c([-1,1])$ be a cut-off function. Moreover, $\Phi$ is normalized such that $\phi=\frac{d^2}{dx^2}\Phi$ and
\begin{align}\label{3.4}
\frac{1}{2\pi}\int_{\mathbb{R}}\phi^2(x)dx=1.
\end{align}
Define the rescalings $\Phi_{\mu}$, $\phi_{\mu}\in C^\infty_c([-1,1])$ as
\begin{align}\label{3.5}
\Phi_{\mu}(x)=\mu^{-\frac{3}{2}}\Phi(\mu x),\quad \phi_{\mu}(x)=\mu^{\frac{1}{2}}\phi(\mu x).
\end{align}
By an abuse of notation, we periodize $\Phi_{\mu}$ and $\phi_{\mu}$ such that they are treated as periodic functions defined on $\mathbb{T}$. By a direct computation, one has $\phi_\mu=\frac{d^2}{dx^2}\Phi_\mu$, and
\begin{align}\label{3.6}
\fint_{\mathbb{T}}\phi_\mu^2(x)dx=1.
\end{align}

For the large parameter $\sigma\in \mathbb{N}$ defined in \eqref{para} and $k\in \Lambda_1\cup\Lambda_2$, we define $\Phi_{(k)}:\mathbb{T}^2\rightarrow \mathbb{R}$ as
\begin{align}\label{3.7}
\Phi_{(k)}(x)=\frac{1}{(N_\Lambda\sigma)^2}\Phi_\mu\left(\sigma N_\Lambda(x-x_k)\cdot\tilde{A_k}\right),
\end{align}
where the vectors $\tilde{k}$ and $\tilde{A_k}$ are defined in Tab. \ref{tab3} and $x_k\in \mathbb{R}^2$ are shifts that ensure that
\begin{align}\label{3.8}
{\rm supp}\Phi_k\cap{\rm supp}\Phi_{k'}=\emptyset\;\;while\;\;k\neq k'\;\;but\;\;\tilde{A_k}=\tilde{A_{k'}}.
\end{align}
Note that such shifts $\{x_k\}$ exist. Indeed, the supports of $\Phi_k$ and $\Phi_{k'}$ are contained in parallel thin rectangles with width $\sim(\mu\sigma)^{-1}$ while $k\neq k'$ but $\tilde{A_k}=\tilde{A_{k'}}$, and we require that $\mu$ and $\sigma$ to be sufficiently large.

Similarly, we define $\phi_{(k)}:\mathbb{T}^2\rightarrow \mathbb{R}$ as
\begin{align}\label{3.9}
\phi_{(k)}(x)=\phi_\mu\left(\sigma N_\Lambda(x-x_k)\cdot\tilde{A_k}\right).
\end{align}
Then a direct computation implies that $\phi_{(k)}=\Delta \Phi_{(k)}$ and
\begin{align}\label{3.10}
\fint_{\mathbb{T}^2}\phi_{(k)}^2 dx=1.
\end{align}

The Mikado flows $W_{(k)}:\mathbb{T}^2\rightarrow\mathbb{R}^3$ are defined as
\begin{align}\label{3.11}
W_{(k)}=\phi_{(k)}k.
\end{align}
Define the skew-symmetric tensors $\Omega_{(k)}\in C_c^\infty(\mathbb{T}^2\rightarrow\mathbb{R}^3\times\mathbb{R}^3)$ as
\begin{align}\label{3.12}
\Omega_{(k)}=k\otimes \nabla\Phi_{(k)}-\nabla\Phi_{(k)}\otimes k.
\end{align}
Then we have the following lemma.
\begin{Lemma}\label{lem3.2}
For the stationary Mikado flows $W_{(k)}$, the following statements holds.\\
(1) Each $W_{(k)}\in C^\infty_c(\mathbb{T}^2)$ is divergence-free, and satisfies
 $${\rm div}\Omega_{(k)}=W_{(k)},\quad{\rm div} (W_{(k)}\otimes W_{(k)})=0,$$
 and
  \begin{align}\label{3.13}
\fint_{\mathbb{T}^2}W_{(k)}\otimes W_{(k)}dx=k\otimes k.
\end{align}
  (2) For any $1\leq p\leq\infty$, we have the following estimates:
  \begin{align}\label{3.14}
\|\nabla^m\phi_{(k)}\|_{L^p(\mathbb{T}^2)}+\|\nabla^mW_{(k)}\|_{L^p(\mathbb{T}^2)}\lesssim_m(\mu\sigma)^m\mu^{\frac{1}{2}-\frac{1}{p}},
\end{align}
\begin{align}\label{3.15}
\|\nabla^m\Omega_{(k)}\|_{L^p(\mathbb{T}^2)}\lesssim_m(\mu\sigma)^{m-1}\mu^{\frac{1}{2}-\frac{1}{p}}.
\end{align}
  (3) For $k\neq k'\in \Lambda_1\cup\Lambda_2$ and $1\leq p\leq\infty$, one has
  \begin{align}\label{3.16}
\|W_{(k)}\otimes W_{(k')}\|_{L^p(\mathbb{T}^2)}\lesssim\mu^{1-\frac{2}{p}}.
\end{align}
\begin{proof}
The first claim can be justified by the fact that $k\cdot(\tilde{A_k},0)=0$, which is obtained by Tab. \ref{tab3} in Appendix, while the second one is similar to \cite[Theorem 4.3]{cl}. Next, we prove the third claim. For $k\neq k'$ but $\tilde{A_k}=\tilde{A_{k'}}$, from \eqref{3.8}, we infer that
\begin{align}\label{3.17}
W_{(k)}\otimes W_{(k')}=0.
\end{align}

For $k\neq k'$ and $\tilde{A_k}\neq\tilde{A_{k'}}$, the set ${\rm supp}W_{(k)}\cap{\rm supp}W_{(k')}$ will be contained in the intersection area as one of the following cases:
 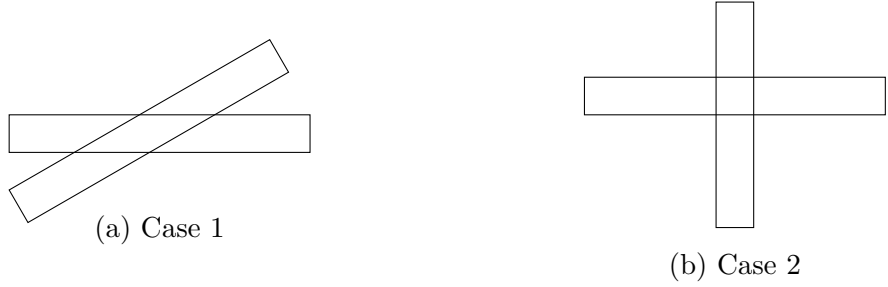
\begin{figure}[H]
    \centering
    \begin{subfigure}[c]{0.45\textwidth}
        \centering
        \begin{tikzpicture}
            \draw (0,1) rectangle (4,0.5);
            \draw[rotate=30] (0,0) rectangle (4,-0.5);
            \node[below] at (2,-0.2) {(a) Case 1};
        \end{tikzpicture}
        \caption*{}
    \end{subfigure}
    \quad\quad
    \begin{subfigure}[c]{0.45\textwidth}
        \centering
        \begin{tikzpicture}
            \draw (0,0) rectangle (4,0.5);
            \draw (1.75,-1.5) rectangle (2.25,1.5);
            \node[below] at (2,-1.7) {(b) Case 2};
        \end{tikzpicture}
        \caption*{}
    \end{subfigure}
    \vspace{-20pt}
    \caption{Typical cases of intersections.}
    \label{fig1}
\end{figure}\vspace{-0.5cm}
\noindent We estimate the supports of $W_{(k)}\otimes W_{(k')}$. As $W_{(k)}\otimes W_{(k')}$ is $(\mathbb{T}/\sigma)^2$-periodic, it suffices to estimate the volume of support on the small squares with side length $2\pi/\sigma$ and then multiply the resulting by $\sigma^2$. On each of these squares, the support of $W_{(k)}\otimes W_{(k')}$ consists of at most $2N_\Lambda$ parallel thin rectangles with length~$\sigma^{-1}$ and width $(\mu\sigma)^{-1}$. Since $\tilde{A_k}\neq\tilde{A_{k'}}$ and the sets $\Lambda_{\alpha}$ are finite, the intersections between different supports of $W_{(k)}$ and $W_{(k')}$ are contained in much smaller rectangles in Figure \ref{fig1} with length and width bounded by $(\mu\sigma)^{-1}$, thus one has
$$\left|{\rm supp}(W_{(k)}\otimes W_{(k')})\right|\lesssim(\mu\sigma)^{-2}\sigma^2\lesssim\mu^{-2}.$$
Then for $1\leq p\leq\infty$ and $\tilde{A_k}\neq\tilde{A_{k'}}$, by \eqref{3.14}, one has
\begin{align}\label{3.18}
\|W_{(k)}\otimes W_{(k')}\|_{L^p}\lesssim\|W_{(k)}\|_{L^\infty}\|W_{(k')}\|_{L^\infty}\left|{\rm supp}(W_{(k)}\otimes W_{(k')})\right|^{\frac{1}{p}}\lesssim\mu^{1-\frac{2}{p}}.
\end{align}
Combining \eqref{3.17} with \eqref{3.18}, we conclude \eqref{3.16}.

\end{proof}
\end{Lemma}

\section{Velocity and magnetic perturbations}\label{sec4}
This section is devoted to the construction of velocity and magnetic perturbations, which consist principle parts and corrector parts. The principle parts are used to cancel the old Reynolds and magnetic stresses, while the corrector parts ensure that the perturbations are divergence-free.
\subsection{Magnetic perturbations}
Let $\chi:\mathbb{R}^3\times\mathbb{R}^3\rightarrow\mathbb{R}^+$ be a smooth function such that $\chi$ is monotonically increasing with respect to $|x|$ and
\begin{align}\label{4.1}
 \chi(x)=\left\{
\begin{aligned}
&2\delta^{-1}\|(\mathring{R}^u_l,\mathring{R}^B_l)\|_{L^\infty_{t}L^1_{x}},\quad 0\leq|x|\leq\|(\mathring{R}^u_l,\mathring{R}^B_l)\|_{L^\infty_{t}L^1_{x}}, \\
&4\delta^{-1}|x|, \qquad\qquad\qquad\quad |x|\geq2\|(\mathring{R}^u_l,\mathring{R}^B_l)\|_{L^\infty_{t}L^1_{x}},\\
\end{aligned}
\right.
\end{align}
where $\delta$ is a small constant defined in Lemma \ref{lem3.1}.

Set $\rho_B(x,t)=\chi(\mathring{R}^B_l)$, and then by the definition of $\chi(x)$, we infer that
\begin{align}\label{4.2}
|\rho_B^{-1}(x,t)\mathring{R}^B_l|\leq\delta,
\end{align}
and
\begin{align}\label{4.3}
\|\rho_B^{\frac{1}{2}}\|_{L^2_x}=\|\rho_B\|^{\frac{1}{2}}_{L^1_x}\lesssim\|(\mathring{R}^u_l,\mathring{R}^B_l)\|^{\frac{1}{2}}_{L^1_x}\lesssim\delta_{q+1}^{\frac{1}{2}},
\end{align}
where the last inequality is given by \eqref{2.9}.

Define $\theta_B$ be a smooth temporal cut-off function as\\
(1) $0\leq\theta_B\leq1$, $\theta_B=1$ on ${\rm supp}_t\mathring{R}^B_l$;\\
(2) ${\rm supp}_t\theta_B\subseteq N_l({\rm supp}_t\mathring{R}^B_l)$;\\
(3) $\|\theta_B\|_{C^N}\lesssim l^{-N}$, for $1\leq N\leq4$.\\
Then we define the coefficient functions as
\begin{align}\label{4.4}
a_{(k)}(x,t)=\theta_B(t)\rho_B^{\frac{1}{2}}(x,t)\gamma_k\left({\rm Id}-\rho_B^{-1}(x,t)\mathring{R}^B_l(x,t)\right),\quad k\in \Lambda_2,
\end{align}
where $\gamma_k$  is the smooth function in Lemma \ref{lem3.1}.

Combining Lemma \ref{lem3.1}, identity \eqref{3.13} and the fact \eqref{4.2}, we have
\begin{align}\label{4.5}
\sum_{k\in\Lambda_2}a_{(k)}^2\fint_{\mathbb{T}^2}W_{(k)}\otimes W_{(k)}dx=\sum_{k\in\Lambda_2}\theta_B^2\rho_B\gamma_k^2\left({\rm Id}-\rho_B^{-1}\mathring{R}^B_l\right)k\otimes k=\theta_B^2\rho_B{\rm Id}-\mathring{R}^B_l.
\end{align}

In view of inequality \eqref{2.8}, \eqref{2.9} and \eqref{4.3}, similar to \cite[Lemma 4.1]{lzz}, we have the following lemma.
\begin{Lemma}\label{lem4.1}
For $1\leq N\leq4$ and $k\in \Lambda_2$, one has
\begin{align}
&\|a_{(k)}\|_{L^2_x}\lesssim\delta_{q+1}^{\frac{1}{2}},\label{4.6}\\
&\|a_{(k)}\|_{C_{t,x}}\lesssim l^{-1},\quad\|a_{(k)}\|_{C^N_{t,x}}\lesssim l^{-4N}.\label{4.7}
\end{align}
\end{Lemma}
Now we are ready to define the magnetic perturbations. The principle part $d_{q+1}^{(p)}$ is given by
\begin{align}\label{4.8}
d_{q+1}^{(p)}:=\sum_{k\in \Lambda_2}a_{(k)}W_{(k)},
\end{align}
and the corresponding incompressibility corrector is given by
\begin{align}\label{4.9}
d_{q+1}^{(c)}:=\sum_{k\in \Lambda_2}\nabla a_{(k)}:\Omega_{(k)}.
\end{align}
Then from Lemma \ref{lem3.2}, one has
\begin{align}\label{4.10}
d_{q+1}^{(p)}+d_{q+1}^{(c)}={\rm div}\sum_{k\in \Lambda_2} a_{(k)}\Omega_{(k)}.
\end{align}
Noting that $a_{(k)}\Omega_{(k)}$ is skew-symmetric matrix, we can infer that ${\rm div}(d_{q+1}^{(p)}+d_{q+1}^{(c)})=0$.

The magnetic perturbation and new magnetic field at level $q+1$ are defined by
\begin{align}
&d_{q+1}=d_{q+1}^{(p)}+d_{q+1}^{(c)},\label{4.11}\\
&B_{q+1}=B_l+d_{q+1},\label{4.12}
\end{align}
then we have ${\rm div}d_{q+1}={\rm div}B_{q+1}=0$.

\subsection{The velocity perturbations}
Inspired by \cite{lzz}, in order to maintain the cancellation between the perturbations and old stresses, we define
\begin{align}\label{4.13}
G^B:=\sum_{k\in \Lambda_2}a_{(k)}^2\fint_{\mathbb{T}^2}W_{(k)}\otimes W_{(k)}dx.
\end{align}
In view of Lemma \ref{lem4.1}, we have
\begin{align}\label{4.14}
\|G^B\|_{L^1_x}\lesssim\delta_{q+1},\quad\|G^B\|_{C_{t,x}}\lesssim l^{-2},\quad \|G^B\|_{C^N_{t,x}}\lesssim l^{-4N-1}.
\end{align}
Set $\rho_u(x,t)=\chi(\mathring{R}^u_l-G^B)$, and by the definition of $\chi(x)$, then we have
\begin{align}\label{4.15}
|\rho_u^{-1}(\mathring{R}^u_l-G^B)|\leq\delta.
\end{align}
From estimates \eqref{2.9} and \eqref{4.14}, one has
\begin{align}\label{4.16}
\|\rho_u^{\frac{1}{2}}\|_{L^2_x}=\|\rho_u\|^{\frac{1}{2}}_{L^1_x}\lesssim\|(\mathring{R}^u_l,\mathring{R}^B_l)\|^{\frac{1}{2}}_{L^1_x}
+\|\mathring{R}^u_l-G^B\|^{\frac{1}{2}}_{L^1_x}\lesssim\delta_{q+1}^{\frac{1}{2}}.
\end{align}

We choose the smooth temporal cut-off function $\theta_u$ as\\
(1) $0\leq\theta_u\leq1$, $\theta_u=1$ on ${\rm supp}_t\mathring{R}^u_l\cup{\rm supp}_tG^B$;\\
(2) ${\rm supp}_t\theta_u\subseteq N_l({\rm supp}_t\mathring{R}^u_l\cup{\rm supp}_tG^B)\subseteq N_{2l}({\rm supp}_t\mathring{R}^u_l\cup{\rm supp}_t\mathring{R}^B_l)$;\\
(3) $\|\theta_u\|_{C^N}\lesssim l^{-N}$, for $1\leq N\leq4$.\\
Then we define the coefficient functions as
\begin{align}\label{4.16'}
a_{(k)}(x,t)=\theta_u(t)\rho_u^{\frac{1}{2}}(x,t)\gamma_k\left({\rm Id}-\rho_u^{-1}(x,t)(\mathring{R}^u_l-G^B)(x,t)\right),\quad k\in \Lambda_1,
\end{align}
where $\gamma_k$ is the smooth function in Lemma \ref{lem3.1}. Employing Lemma \ref{lem3.1}, \eqref{3.13} and  \eqref{4.15}, we have
\begin{align}\label{4.17}
\sum_{k\in\Lambda_1}a_{(k)}^2\fint_{\mathbb{T}^2}W_{(k)}\otimes W_{(k)}dx&=\sum_{k\in\Lambda_1}\theta_u^2\rho_u\gamma_k^2\left({\rm Id}-\rho_u^{-1}(\mathring{R}^u_l-G^B)\right)k\otimes k\notag\\
&=\theta_u^2\rho_u{\rm Id}-\mathring{R}^u_l+G^B.
\end{align}

In view of estimates \eqref{2.8}, \eqref{2.9}, \eqref{4.14} and \eqref{4.16}, similar to \cite[Lemma 4.2]{lzz}, we have
\begin{Lemma}\label{lem4.2}
For $1\leq N\leq4$ and $k\in \Lambda_1$, one has
\begin{align}
&\|a_{(k)}\|_{L^2_x}\lesssim\delta_{q+1}^{\frac{1}{2}},\label{4.18}\\
&\|a_{(k)}\|_{C_{t,x}}\lesssim l^{-1},\quad\|a_{(k)}\|_{C^N_{t,x}}\lesssim l^{-8N}.\label{4.19}
\end{align}
\end{Lemma}

Now we are ready to define the velocity perturbations. The principle part $w_{q+1}^{(p)}$ is given by
\begin{align}\label{4.20}
w_{q+1}^{(p)}:=\sum_{k\in\Lambda_1}a_{(k)}W_{(k)},
\end{align}
and the corresponding incompressibility corrector is given by
\begin{align}\label{4.21}
w_{q+1}^{(c)}:=\sum_{k\in \Lambda_1}\nabla a_{(k)}:\Omega_{(k)}.
\end{align}
Then from Lemma \ref{lem3.2}, one has
\begin{align}\label{4.22}
w_{q+1}^{(p)}+w_{q+1}^{(c)}={\rm div}\sum_{k\in \Lambda_1} a_{(k)}\Omega_{(k)}.
\end{align}
Noting that $a_{(k)}\Omega_{(k)}$ is skew-symmetric matrix, we infer that ${\rm div}(w_{q+1}^{(p)}+w_{q+1}^{(c)})=0$.

The velocity perturbation and new velocity field at level $q+1$ are defined by
\begin{align}
&w_{q+1}=w_{q+1}^{(p)}+w_{q+1}^{(c)},\label{4.23}\\
&u_{q+1}=u_l+w_{q+1},\label{4.24}
\end{align}
then we have ${\rm div}w_{q+1}={\rm div}u_{q+1}=0$.

In the end of this subsection, we use the principle parts of perturbations to cancel the old magnetic and Reynolds stresses $\mathring{R}^B_l$ and $\mathring{R}^u_l$. By the definition \eqref{4.8} and equality \eqref{4.5}, we have
\begin{align}\label{4.25}
&d_{q+1}^{(p)}\otimes d_{q+1}^{(p)}+\mathring{R}^B_l\notag\\
&=\sum_{k\in \Lambda_2}a_{(k)}^2\fint_{\mathbb{T}^2}\left(W_{(k)}\otimes W_{(k)}\right)dx+\mathring{R}^B_l+\sum_{k\in \Lambda_2}a_{(k)}^2\mathbb{P}_{\neq0}\left(W_{(k)}\otimes W_{(k)}\right)\notag\\
&\quad\quad+\sum_{k\neq k'\in \Lambda_2}a_{(k)}a_{(k')}W_{(k)}\otimes W_{(k')}\notag\\
&=\theta_B^2\rho_B{\rm Id}+\sum_{k\in \Lambda_2}a_{(k)}^2\mathbb{P}_{\neq0}\left(W_{(k)}\otimes W_{(k)}\right)+\sum_{k\neq k'\in \Lambda_2}a_{(k)}a_{(k')}W_{(k)}\otimes W_{(k')}
\end{align}
Similarly, from the definitions \eqref{4.8} and \eqref{4.20}, and equality \eqref{4.17}, one has
\begin{align}\label{4.26}
&w_{q+1}^{(p)}\otimes w_{q+1}^{(p)}-d_{q+1}^{(p)}\otimes d_{q+1}^{(p)}+\mathring{R}^u_l\notag\\
&=\sum_{k\in \Lambda_1}a_{(k)}^2\fint_{\mathbb{T}^2}\left(W_{(k)}\otimes W_{(k)}\right)dx-\sum_{k\in \Lambda_2}a_{(k)}^2\fint_{\mathbb{T}^2}\left(W_{(k)}\otimes W_{(k)}\right)dx+\mathring{R}^u_l\notag\\
&\quad+\sum_{k\in \Lambda_1}a_{(k)}^2\mathbb{P}_{\neq0}\left(W_{(k)}\otimes W_{(k)}\right)-\sum_{k\in \Lambda_2}a_{(k)}^2\mathbb{P}_{\neq0}\left(W_{(k)}\otimes W_{(k)}\right)\notag\\
&\quad+\sum_{k\neq k'\in \Lambda_1}a_{(k)}a_{(k')}W_{(k)}\otimes W_{(k')}-\sum_{k\neq k'\in \Lambda_2}a_{(k)}a_{(k')}W_{(k)}\otimes W_{(k')}\notag\\
&=\theta_u^2\rho_u{\rm Id}+\sum_{k\in \Lambda_1}a_{(k)}^2\mathbb{P}_{\neq0}\left(W_{(k)}\otimes W_{(k)}\right)-\sum_{k\in \Lambda_2}a_{(k)}^2\mathbb{P}_{\neq0}\left(W_{(k)}\otimes W_{(k)}\right)\notag\\
&\quad+\sum_{k\neq k'\in \Lambda_1}a_{(k)}a_{(k')}W_{(k)}\otimes W_{(k')}-\sum_{k\neq k'\in \Lambda_2}a_{(k)}a_{(k')}W_{(k)}\otimes W_{(k')}
\end{align}


\subsection{Estimates of the perturbations}\label{sec4.3}
In order to obtain the inductive estimate \eqref{1.15}, we need the following decorrelation lemma, see \cite[Lemma 2.4]{cl1}.
\begin{Lemma}\label{lem4.3}
Let $\theta\in \mathbb{N}$ and $f,g:\mathbb{T}^2\rightarrow\mathbb{R}$ be smooth functions. Then for every $p\in[1,\infty]$,
\begin{align}\label{4.27}
\left|\|fg(\theta\cdot)\|_{L^p(\mathbb{T}^2)}-\|f\|_{L^p(\mathbb{T}^2)}\|g\|_{L^p(\mathbb{T}^2)}\right|\lesssim\theta^{-\frac{1}{p}}\|f\|_{C^1(\mathbb{T}^2)}\|g\|_{L^p(\mathbb{T}^2)}.
\end{align}
\end{Lemma}

The estimates of the velocity and magnetic perturbations are summarized in the following proposition.
\begin{Proposition}\label{prop4.4}
Let $s\geq0$, and $1\leq r\leq\infty$, for any integer $0\leq N\leq4$, the following estimates holds
\begin{align}
&\|(w_{q+1}^{(p)},d_{q+1}^{(p)})\|_{L^2_x}\lesssim\delta_{q+1}^{\frac{1}{2}},\label{4.28}\\
&\|(w_{q+1}^{(c)},d_{q+1}^{(c)})\|_{L^2_x}\lesssim\lambda_{q+1}^{-1},\label{4.29}\\
&\|(w_{q+1}^{(p)},d_{q+1}^{(p)})\|_{W_x^{N,r}}\lesssim l^{-8N}\lambda_{q+1}^{N(1+\varepsilon)+\frac{1}{2}-\frac{1}{r}},\label{4.30'}\\
&\|(w_{q+1}^{(c)},d_{q+1}^{(c)})\|_{C^N_{t,x}}\lesssim\lambda_{q+1}^{N(1+\varepsilon)-\frac{1}{2}},\label{4.31''}\\
&\|(w_{q+1},d_{q+1}\|_{C^N_{t,x}}\lesssim\lambda_{q+1}^{\frac{3}{2}N+1},\label{4.30}\\
&\|(w_{q+1},d_{q+1})\|_{W_x^{s,r}}\lesssim l^{-8s}\lambda_{q+1}^{s(1+\varepsilon)+\frac{1}{2}-\frac{1}{r}}.\label{4.31'}
\end{align}
\end{Proposition}
\begin{proof}
We apply Lemma \ref{lem4.3} with $f=a_{(k)}$, $g=\phi_{(k)}$ and $\theta=\sigma=\lambda_{q+1}^\varepsilon$. By estimate \eqref{3.14} and Lemma \ref{4.1} and \ref{4.2}, we infer
\begin{align*}
\|(w_{q+1}^{(p)},d_{q+1}^{(p)})\|_{L^2_x}&\lesssim\sum_{k\in \Lambda_1\cup\Lambda_2}\|a_{(k)}\|_{L^2_x}\|\phi_{(k)}\|_{L^2_x}+\lambda_{q+1}^{-\frac{1}
{2}\varepsilon}\|a_{(k)}\|_{C^1_x}\|\phi_{(k)}\|_{L^2_x}\\
&\lesssim \delta_{q+1}^{\frac{1}{2}}+l^{-8}\lambda_{q+1}^{-\frac{1}{2}}\lesssim \delta_{q+1}^{\frac{1}{2}},
\end{align*}
where the last inequality is obtained by \eqref{1.9} and \eqref{1.10}. Now, we prove \eqref{4.29}. By using the definitions \eqref{4.9} and \eqref{4.21}, inequality \eqref{3.15}, Lemmas \ref{lem4.1} and \ref{lem4.2}, we have
\begin{align*}
\|(w_{q+1}^{(c)},d_{q+1}^{(c)})\|_{L^2_x}\lesssim\sum_{k\in \Lambda_1\cup\Lambda_2}\|\nabla a_{(k)}\|_{C_x}\|\Omega_{(k)}\|_{L^2_x}
\lesssim l^{-8}(\mu\sigma)^{-1}\lesssim \lambda_{q+1}^{-1}.
\end{align*}
Next, we justify inequality \eqref{4.30'}. Utilizing the definitions \eqref{4.8} and \eqref{4.20}, inequality \eqref{3.14}, Lemma \ref{lem4.1} and \ref{lem4.2}, one has
\begin{align}\label{4.31}
\|(w_{q+1}^{(p)},d_{q+1}^{(p)})\|_{W_x^{N,r}}&\lesssim\sum_{k\in \Lambda_1\cup\Lambda_2}\|a_{(k)}\|_{C^N_{t,x}}\|W_{(k)}\|_{W^{N,r}_x}\notag\\
&\lesssim l^{-8N}(\mu\sigma)^{N}\mu^{\frac{1}{2}-\frac{1}{r}}\lesssim l^{-8N}\lambda_{q+1}^{N(1+\varepsilon)}\lambda_{q+1}^{\frac{1}{2}-\frac{1}{r}},
\end{align}
where the last inequality is obtained by \eqref{1.9}--\eqref{1.11}.
Similarly, by the definitions \eqref{4.9} and \eqref{4.21}, inequality \eqref{3.15}, Lemma \ref{lem4.1} and \ref{lem4.2}, we infer that
\begin{align}\label{4.32}
\|(w_{q+1}^{(c)},d_{q+1}^{(c)})\|_{C^N_{t,x}}&\lesssim\sum_{k\in \Lambda_1\cup\Lambda_2}\|\nabla a_{(k)}\|_{C^N_{t,x}}\|\Omega_{(k)}\|_{C^N_{x}}\notag\\
&\lesssim l^{-8N}(\mu\sigma)^{N-1}\mu^{\frac{1}{2}}\lesssim \lambda_{q+1}^{-\frac{1}{2}+N(\varepsilon+1)},
\end{align}
then \eqref{4.31''} is verified. Recalling the definitions of $w_{q+1}$ and $d_{q+1}$ in \eqref{4.10} and \eqref{4.23}, respectively, estimates \eqref{4.31} and \eqref{4.32} imply that \eqref{4.30}. By estimates \eqref{4.31} and \eqref{4.32}, we have
\begin{align*}
\|w_{q+1}\|_{W^{N,r}_x}&\lesssim\|w^{(p)}_{q+1}\|_{W^{N,r}_x}+\|w^{(c)}_{q+1}\|_{W^{N,r}_x}\\
&\lesssim\|w^{(p)}_{q+1}\|_{W^{N,r}_x}+\|w^{(c)}_{q+1}\|_{C^{N}_{t,x}}\\
&\lesssim l^{-8N}(\mu\sigma)^{N}\mu^{\frac{1}{2}-\frac{1}{r}}+l^{-8(N+1)}(\mu\sigma)^{N-1}\mu^{\frac{1}{2}}\\
&\lesssim l^{-8N}\lambda_{q+1}^{N(1+\varepsilon)}\lambda_{q+1}^{\frac{1}{2}-\frac{1}{r}}.
\end{align*}
Similarly, one has $\|d_{q+1}\|_{W^{N,r}_x}\lesssim l^{-4N}\lambda_{q+1}^{N(1+\varepsilon)}\lambda_{q+1}^{\frac{1}{2}-\frac{1}{r}}$, then \eqref{4.31'} is obtained by the interpolation inequality.
\end{proof}

At the end of this section, we justify the inductive estimates \eqref{1.12} and \eqref{1.15} at level $q+1$. By the definition of $u_{q+1}$ in \eqref{4.24}, standard mollification estimates and inequalities \eqref{1.12}, \eqref{4.28}, \eqref{4.29} and \eqref{4.30}, we have
\begin{align}\label{4.33}
\|u_{q+1}\|_{C^1_{t,x}}\lesssim\|u_{l}\|_{C^1_{t,x}}+\|w_{q+1}\|_{C^1_{t,x}}\lesssim\|u_{q}\|_{C^1_{t,x}}+\lambda_{q+1}^{\frac{5}{2}}\lesssim\lambda_{q+1}^4,
\end{align}
\begin{align}\label{4.34}
\|u_{q+1}-u_q\|_{L^2_x}&\leq\|u_{q+1}-u_{l}\|_{L^2_x}+\|u_{l}-u_q\|_{L^2_x}\notag\\
&\lesssim\|w_{q+1}\|_{L^2_x}+\|u_{l}-u_q\|_{C_{t,x}}\notag\\
&\lesssim\delta_{q+1}^{\frac{1}{2}}+l\|u_q\|_{C^1_{t,x}}\lesssim\delta_{q+1}^{\frac{1}{2}}+\lambda_q^{-13}\lesssim\delta_{q+1}^{\frac{1}{2}},
\end{align}
where the last inequality is given by the fact that $\beta\ll\frac{1}{b}$, which is deduced from \eqref{1.9} and \eqref{1.10}. A same argument implies that
\begin{align}
&\|B_{q+1}\|_{C^1_{t,x}}\lesssim\lambda_{q+1}^4,\label{4.35}\\
&\|u_{q+1}-u_q\|_{L^2_x}\lesssim\delta_{q+1}^{\frac{1}{2}}.\label{4.36}
\end{align}
Thus inductive estimates \eqref{1.12} and \eqref{1.15} are verified, and \eqref{1.16} is obtained by \eqref{4.31'} with $r=1$ and $s=0$.

\section{Reynolds and magnetic stresses}\label{sec5}
In this section, we will complete the proof of Theorem \ref{main iteration}. The inverses of divergence and curl operators are critical for the proof hereinbelow. The inverse-divergence operator $\mathcal{R}$ is defined by
$$(\mathcal{R}u)^{k\ell}=(\partial_k\Delta^{-1}u^\ell+\partial_\ell\Delta^{-1}u^k)-\frac{1}{2}(\delta_{k\ell}+\partial_k\partial_\ell\Delta^{-1}){\rm div}\Delta^{-1}u,$$
for $\int_{\mathbb{T}^2}udx=0$. The matrix $\mathcal{R}u$ is symmetric and traceless and one has ${\rm div}(\mathcal{R}u)=u$. We abuse the notation $\mathcal{R}u:=\mathcal{R}(u-\int_{\mathbb{T}^2}udx)$ for vector field $u$ with $\int_{\mathbb{T}^2}udx\neq0$. Moreover, $|\nabla|\mathcal{R}$ is Calderon--Zygmund operator, thus it is bounded in $L^p$ for $1<p<\infty$. See \cite{bcv} for more details.

The inverse of curl is defined by
$${\rm curl}^{-1}f=(-\Delta)^{-1}\nabla\times f \quad {\rm for}\quad {\rm div}f=0,$$
then one has $\nabla\times({\rm curl}^{-1})f=f$ and ${\rm curl}^{-1}(\nabla\times f)=f$ while ${\rm div}f=0$.
\subsection{Decomposition of Reynolds and magnetic stresses}
Firstly, we focus on the decomposition of magnetic stress. We hope that the new pair $(u_{q+1},B_{q+1},\mathring{R}_{q+1}^u,\mathring{R}_{q+1}^B)$ satisfies \eqref{Hall-MHDapp} with $q+1$ replacing $q$, thus by the definitions of $u_{q+1}$, $B_{q+1}$, $w_{q+1}$ and $d_{q+1}$ in \eqref{4.11}, \eqref{4.12}, \eqref{4.23} and \eqref{4.24}, and equation \eqref{Hall-MHDmollify}, we can deduce  the equation for the magnetic stress:
\begin{align}\label{5.1}
&\nabla\times{\rm div}\mathring{R}_{q+1}^B\notag\\
&=\partial_td_{q+1}+\nu_2(-\Delta)^{\alpha_2} d_{q+1}+{\rm div}(B_{l}\otimes w_{q+1}+d_{q+1}\otimes u_{l}\notag\\
&\quad\underbrace{\quad\quad-u_{l}\otimes d_{q+1}-w_{q+1}\otimes B_{l})+\nabla\times{\rm div}(B_{l}\otimes d_{q+1}+d_{q+1}\otimes B_{l})}_{\nabla\times{\rm div}\mathring{R}_{lin}^B}\notag\\
&\quad+\nabla\times{\rm div}(d^{(p)}_{q+1}\otimes d^{(c)}_{q+1}+d^{(c)}_{q+1}\otimes d_{q+1})\notag\\
&\quad\underbrace{\quad\quad+{\rm div}(d^{(p)}_{q+1}\otimes w^{(c)}_{q+1}+d^{(c)}_{q+1}\otimes w_{q+1}-w_{q+1}\otimes d^{(c)}_{q+1}-w^{(c)}_{q+1}\otimes d^{(p)}_{q+1})}_{\nabla\times{\rm div}\mathring{R}_{cor}^B}\notag\\
&\quad\underbrace{+\nabla\times{\rm div}(d^{(p)}_{q+1}\otimes d^{(p)}_{q+1}+\mathring{R}_{l}^B)+{\rm div}(d^{(p)}_{q+1}\otimes w^{(p)}_{q+1}-w^{(p)}_{q+1}\otimes d^{(p)}_{q+1})}_{\nabla\times{\rm div}\mathring{R}_{osc}^B}\notag\\
&\quad\underbrace{+\nabla\times{\rm div}(\mathring{R}_{com,1}^B)+{\rm div}\mathring{R}_{com,2}^B}_{\nabla\times{\rm div}\mathring{R}_{com}^B}.
\end{align}
Thus, the magnetic stress is defined as
\begin{align}\label{5.2'}
\mathring{R}_{q+1}^B=\mathring{R}_{lin}^B+\mathring{R}_{cor}^B+\mathring{R}_{osc}^B+\mathring{R}_{com}^B,
\end{align}
where the linear error
\begin{align}\label{5.2}
\mathring{R}_{lin}^B&=\mathcal{R}{\rm curl}^{-1}(\partial_td_{q+1}+\nu_2(-\Delta)^{\alpha_2} d_{q+1}+{\rm div}(B_{l}\otimes w_{q+1}+d_{q+1}\otimes u_{l}\notag\\
&\quad-u_{l}\otimes d_{q+1}-w_{q+1}\otimes B_{l}))+\mathcal{R}{\rm div}(B_{l}\otimes d_{q+1}+d_{q+1}\otimes B_{l}),
\end{align}
the corrector error
\begin{align}\label{5.3}
\mathring{R}_{cor}^B&=\mathcal{R}{\rm curl}^{-1}{\rm div}(d^{(p)}_{q+1}\otimes w^{(c)}_{q+1}+d^{(c)}_{q+1}\otimes w_{q+1}-w_{q+1}\otimes d^{(c)}_{q+1}-w^{(c)}_{q+1}\otimes d^{(p)}_{q+1})\notag\\
&\quad+\mathcal{R}{\rm div}(d^{(p)}_{q+1}\otimes d^{(c)}_{q+1}+d^{(c)}_{q+1}\otimes d_{q+1}),
\end{align}
the commutator error
\begin{align}\label{5.4}
\mathring{R}_{com}^B=\mathring{R}_{com,1}^B+\mathcal{R}{\rm curl}^{-1}{\rm div}\mathring{R}_{com,2}^B,
\end{align}
and the oscillation error
\begin{align}\label{5.5}
\mathring{R}_{osc}^B&=\sum_{k\in \Lambda_2}\mathcal{R}\mathbb{P}_{\neq0}\left(\nabla(a_{(k)}^2)\mathbb{P}_{\neq0}(W_{(k)}\otimes W_{(k)})\right)+\sum_{k\neq k'\in \Lambda_2}\mathcal{R}{\rm div}\left(a_{(k)}a_{(k')}W_{(k)}\otimes W_{(k')}\right))\notag\\
&\quad+\sum_{k\in \Lambda_2, k'\in \Lambda_1}\mathcal{R}{\rm curl}^{-1}{\rm div}\left(a_{(k)}a_{(k')}(W_{(k)}\otimes W_{(k')}-W_{(k')}\otimes W_{(k)})\right).
\end{align}
The oscillation error is obtained by \eqref{4.25} and Lemma \ref{lem3.2}, while $\mathring{R}_{com,1}$ and $\mathring{R}_{com,2}$ are defined in \eqref{2.4} and \eqref{2.5}, respectively. Noting that ${\rm div}{\rm div}S=0$ for skew-symmetric matrix $S$, the operator ${\rm curl}^{-1}$ in the definitions of $\mathring{R}_{lin}^B$, $\mathring{R}_{cor}^B$, $\mathring{R}_{osc}^B$ and $\mathring{R}_{com,2}^B$ makes sense.

Similar to \eqref{5.1}, we deduce the equation of the Reynolds stress:
\begin{align}\label{5.6}
&{\rm div}\mathring{R}_{q+1}^u-\nabla p_{q+1}+\nabla p_l\notag\\
&=\underbrace{\partial_tw_{q+1}+\nu_1(-\Delta)^{\alpha_1} w_{q+1}+{\rm div}(u_{l}\otimes w_{q+1}+w_{q+1}\otimes u_{l}-B_{l}\otimes d_{q+1}-d_{q+1}\otimes B_{l})}_{{\rm div}\mathring{R}_{lin}^u+\nabla p_{lin}}\notag\\
&\quad\underbrace{+{\rm div}(w^{(p)}_{q+1}\otimes w^{(c)}_{q+1}+w^{(c)}_{q+1}\otimes w_{q+1}-d_{q+1}^{(p)}\otimes d^{(c)}_{q+1}-d^{(c)}_{q+1}\otimes d_{q+1})}_{{\rm div}\mathring{R}_{cor}^u+\nabla p_{cor}}\notag\\
&\quad\underbrace{+{\rm div}(w^{(p)}_{q+1}\otimes w^{(p)}_{q+1}-d^{(p)}_{q+1}\otimes d^{(p)}_{q+1}+\mathring{R}_{l}^u)}_{{\rm div}\mathring{R}_{osc}^u+\nabla p_{osc}}\notag\\
&\quad+{\rm div}\mathring{R}_{com}^u,
\end{align}
where the commutator error $\mathring{R}_{com}^u$ is defined in \eqref{2.3}.
Then the Reynolds stress $\mathring{R}_{q+1}^u$ is defined by
\begin{align}\label{5.7}
\mathring{R}_{q+1}^u=\mathring{R}_{lin}^u+\mathring{R}_{cor}^u+\mathring{R}_{osc}^u+\mathring{R}_{com'}^u,
\end{align}
where the linear error
\begin{align}\label{5.8}
\mathring{R}_{lin}^u=&\mathcal{R}\left(\partial_tw_{q+1}+\nu_1(-\Delta)^{\alpha_1} w_{q+1}\right)\notag\\
&+\mathcal{R}\mathbb{P}_{H}{\rm div}(u_{l}\otimes w_{q+1}+w_{q+1}\otimes u_{l}-B_{l}\otimes d_{q+1}-d_{q+1}\otimes B_{l}),
\end{align}
the corrector error
\begin{align}\label{5.9}
\mathring{R}_{cor}^u=\mathcal{R}\mathbb{P}_{H}{\rm div}\left(w^{(p)}_{q+1}\otimes w^{(c)}_{q+1}+w^{(c)}_{q+1}\otimes w_{q+1}-d_{q+1}^{(p)}\otimes d^{(c)}_{q+1}-d^{(c)}_{q+1}\otimes d_{q+1}\right),
\end{align}
the oscillation error
\begin{align}\label{5.10}
\mathring{R}_{osc}^u=&\left(\sum_{k\in \Lambda_1}-\sum_{k\in \Lambda_2}\right)\mathcal{R}\mathbb{P}_H\mathbb{P}_{\neq0}\left(\nabla(a_{(k)}^2)\mathbb{P}_{\neq0}(W_{(k)}\otimes W_{(k)})\right)\notag\\
&+\left(\sum_{k\neq k'\in \Lambda_1}-\sum_{k\neq k'\in \Lambda_2}\right)\mathcal{R}\mathbb{P}_H{\rm div}\left(a_{(k)}a_{(k')}(W_{(k)}\otimes W_{(k')})\right),
\end{align}
and the commutator error
\begin{align}\label{5.11}
\mathring{R}_{com'}^u=\mathcal{R}\mathbb{P}_H{\rm div}\mathring{R}_{com}^u.
\end{align}
The oscillation error $\mathring{R}_{osc}^u$ is obtained by \eqref{4.26} and Lemma \ref{lem3.2}.

In view of ${\rm div}\mathcal{R}={\rm Id}$, we have
\begin{align}\label{5.12}
{\rm div}\mathring{R}_{q+1}^u=\mathbb{P}_H\left(\partial_tu_{q+1}+\nu_1(-\Delta)^{\alpha_1} u_{q+1}+{\rm div}(u_{q+1}\otimes u_{q+1}-B_{q+1}\otimes B_{q+1})\right),
\end{align}
which implies that $\mathring{R}_{q+1}^u$ satisfies the velocity equation in \eqref{Hall-MHDapp}. Moreover, one has
\begin{align}\label{5.13}
\mathring{R}_{q+1}^u=\mathcal{R}\mathbb{P}_H{\rm div}\mathring{R}_{q+1}^u.
\end{align}

\subsection{Estimates of Reynolds and magnetic stresses}\label{sec5.2}

The purpose of this subsection is justifying that the Reynolds and magnetic stresses $(\mathring{R}_{q+1}^u,\mathring{R}_{q+1}^B)$ satisfies the inductive estimate \eqref{1.14} at level $q+1$.

\underline{Linear errors.} Firstly, we estimate the linear errors  $\mathring{R}_{lin}^B$ and $\mathring{R}_{lin}^u$ defined in \eqref{5.2} and \eqref{5.8}, respectively. For $1<p<2$ with
\begin{align}\label{5.14'}
1-\frac{1}{p}<\frac{1}{8}\varepsilon,
\end{align}
we have
\begin{align}\label{5.14}
\|\mathring{R}_{lin}^B\|_{L^p}&\lesssim\|\mathcal{R}{\rm curl}^{-1}\partial_td_{q+1}\|_{L^p}+\|\mathcal{R}{\rm curl}^{-1}(-\Delta)^{\alpha_2} d_{q+1}\|_{L^p}\notag\\
&\quad+\|\mathcal{R}{\rm curl}^{-1}{\rm div}(B_{l}\otimes w_{q+1}+d_{q+1}\otimes u_{l}-u_{l}\otimes d_{q+1}-w_{q+1}\otimes B_{l})\|_{L^p}\notag\\
&\quad+\|\mathcal{R}{\rm div}(B_{l}\otimes d_{q+1}+d_{q+1}\otimes B_{l})\|_{L^p}\notag\\
&=:I_{lin,1}+I_{lin,2}+I_{lin,3}+I_{lin,4},
\end{align}
and
\begin{align}\label{5.15}
\|\mathring{R}_{lin}^u\|_{L^p}&\lesssim\|\mathcal{R}\partial_tw_{q+1}\|_{L^p}+\|\mathcal{R}(-\Delta)^{\alpha_1} w_{q+1}\|_{L^p}\notag\\
&\quad+\|\mathcal{R}\mathbb{P}_H{\rm div}(u_{l}\otimes w_{q+1}+w_{q+1}\otimes u_{l}-B_{l}\otimes d_{q+1}-d_{q+1}\otimes B_{l})\|_{L^p}\notag\\
&=:J_{lin,1}+J_{lin,2}+J_{lin,3}.
\end{align}
Inserting \eqref{4.10} and \eqref{4.11} into $I_{lin,1}$, by estimates \eqref{3.15} and \eqref{4.7}, one has
\begin{align}\label{5.16}
I_{lin,1}&=\|\mathcal{R}{\rm curl}^{-1}\partial_t{\rm div}\sum_{k\in \Lambda_2}a_{(k)}\Omega_{(k)}\|_{L^p}\notag\\
&\lesssim\sum_{k\in \Lambda_2}\||\nabla|^{-1}\mathbb{P}_{\neq0}\left((\partial_t a_{(k)})\Omega_{(k)}\right)\|_{L^p}\notag\\
&\lesssim\sum_{k\in \Lambda_2}\|\partial_t a_{(k)}\|_{C_x}\|\Omega_{(k)}\|_{L^p}\notag\\
&\lesssim l^{-4}\mu^{\frac{1}{2}-\frac{1}{p}}\lesssim\lambda_{q+1}^{-\frac{1}{4}\varepsilon},
\end{align}
where we have used the Calderon-Zygmund inequality and the following estimate (see \cite[Lemma B.1]{bv} for example):
\begin{align}\label{5.17}
\||\nabla|^{-1}\mathbb{P}_{\neq0}\|_{L^p\rightarrow L^p}\lesssim1.
\end{align}
Similarly, by \eqref{4.22}, \eqref{4.23}, the Calderon-Zygmund inequality and estimates \eqref{3.15}, \eqref{4.19}, inequality \eqref{5.14'}, we can infer that
\begin{align}\label{5.18}
J_{lin,1}&=\|\mathcal{R}\partial_t{\rm div}\sum_{k\in \Lambda_1}a_{(k)}\Omega_{(k)}\|_{L^p}\notag\\
&\lesssim\sum_{k\in \Lambda_1}\|\partial_t a_{(k)}\|_{C_x}\|\Omega_{(k)}\|_{L^p}\notag\\
&\lesssim l^{-8}\mu^{\frac{1}{2}-\frac{1}{p}}\lesssim\lambda_{q+1}^{-\frac{1}{4}\varepsilon}.
\end{align}

By estimate \eqref{4.31'}, inequality \eqref{1.11} and \eqref{5.14'}, we have
\begin{align}\label{5.19}
I_{lin,2}&\lesssim\|d_{q+1}\|_{W^{2\alpha_2-2,p}}\notag\\
&\lesssim l^{-8N}\lambda_{q+1}^{(1+\varepsilon)(2\alpha_2-2)+\frac{1}{2}-\frac{1}{p}}\lesssim l^{-8N}\lambda_{q+1}^{(1+\varepsilon)(\frac{1}{2}-8\varepsilon)-\frac{1}{2}+\frac{1}{8}\varepsilon}\lesssim\lambda_{q+1}^{-4\varepsilon},
\end{align}
and
\begin{align}\label{5.20}
J_{lin,2}&\lesssim\|w_{q+1}\|_{W^{2\alpha_1-1,p}}\notag\\
&\lesssim l^{-8N}\lambda_{q+1}^{(1+\varepsilon)(2\alpha_1-1)+\frac{1}{2}-\frac{1}{p}}\lesssim l^{-8N}\lambda_{q+1}^{(1+\varepsilon)(\frac{1}{2}-8\varepsilon)-\frac{1}{2}+\frac{1}{8}\varepsilon}\lesssim\lambda_{q+1}^{-4\varepsilon}.
\end{align}

Applying the Calderon-Zygmund inequality, estimates \eqref{1.12}, \eqref{4.31'}, \eqref{5.17} and inequality \eqref{5.14'}, one has
\begin{align}\label{5.21}
I_{lin,3}+I_{lin,4}&\lesssim\||\nabla|^{-1}\mathbb{P}_{\neq0}\left(B_l\otimes w_{q+1}+d_{q+1}\otimes u_l-u_l\otimes d_{q+1}-w_{q+1}\otimes B_l\right)\|_{L^p}\notag\\
&\quad+\|B_l\|_{L^\infty}\|d_{q+1}\|_{L^p}\notag\\
&\lesssim \|(u_l,B_l)\|_{L^\infty}\|(w_{q+1},d_{q+1})\|_{L^p}\notag\\
&\lesssim \lambda_{q}^4\lambda_{q+1}^{\frac{1}{2}-\frac{1}{p}}\lesssim\lambda_{q+1}^{-\frac{1}{4}\varepsilon},
\end{align}
and similarly,
\begin{align}\label{5.22}
J_{lin,3}\lesssim\|(u_l,B_l)\|_{L^\infty}\|(w_{q+1},d_{q+1})\|_{L^p}\lesssim\lambda_{q+1}^{-\frac{1}{4}\varepsilon}.
\end{align}
Then inserting estimates \eqref{5.16} and \eqref{5.18}--\eqref{5.22} into \eqref{5.14} and \eqref{5.15}, one has
\begin{align}\label{5.23}
\|(\mathring{R}_{lin}^u,\mathring{R}_{lin}^B)\|_{L^p}\lesssim\lambda_{q+1}^{-\frac{1}{4}\varepsilon}.
\end{align}
\underline{Oscillation errors.}
We need the following lemma (see \cite[Lemma 7.4]{lq}) to estimate oscillation errors.
\begin{Lemma}\label{lem5.1}
For any given $1<p<\infty$, $\lambda\in \mathbb{Z}_+$, $a\in C^2(\mathbb{T}^2;\mathbb{R})$ and $f\in L^p(\mathbb{T}^2)$, one has
$$\||\nabla|^{-1}\mathbb{P}_{\neq0}(a\mathbb{P}_{\geq\lambda}f)\|_{L^p}\lesssim \lambda^{-1}\|a\|_{C^2}\|f\|_{L^p}.$$
\end{Lemma}

Recalling the definitions of oscillation errors $\mathring{R}_{osc}^B$ and $\mathring{R}_{osc}^u$ in \eqref{5.5} and \eqref{5.10}, we can decompose
\begin{align}
\mathring{R}_{osc}^B=\mathring{R}_{osc,x}^B+\mathring{R}_{osc,far}^B,\label{5.24}\\
\mathring{R}_{osc}^u=\mathring{R}_{osc,x}^u+\mathring{R}_{osc,far}^u,\label{5.25}
\end{align}
where
\begin{align*}
&\mathring{R}_{osc,x}^B=\sum_{k\in \Lambda_2}\mathcal{R}\mathbb{P}_{\neq0}\left(\nabla(a_{(k)}^2)\mathbb{P}_{\neq0}(W_{(k)}\otimes W_{(k)})\right),\\
&\mathring{R}_{osc,far}^B=\sum_{k\neq k'\in \Lambda_2}\mathcal{R}{\rm div}\left(a_{(k)}a_{(k')}W_{(k)}\otimes W_{(k')}\right)\\
&\qquad\qquad+\sum_{k\in \Lambda_2, k'\in \Lambda_1}\mathcal{R}{\rm curl}^{-1}{\rm div}\left(a_{(k)}a_{(k')}(W_{(k)}\otimes W_{(k')}-W_{(k')}\otimes W_{(k)})\right),\\
&\mathring{R}_{osc,x}^u=\left(\sum_{k\in \Lambda_1}-\sum_{k\in \Lambda_2}\right)\mathcal{R}\mathbb{P}_H\mathbb{P}_{\neq0}\left(\nabla(a_{(k)}^2)\mathbb{P}_{\neq0}(W_{(k)}\otimes W_{(k)})\right),\\
&\mathring{R}_{osc,far}^u=\left(\sum_{k\neq k'\in \Lambda_1}-\sum_{k\neq k'\in \Lambda_2}\right)\mathcal{R}\mathbb{P}_H{\rm div}\left(a_{(k)}a_{(k')}(W_{(k)}\otimes W_{(k')})\right).
\end{align*}

For the terms $\mathring{R}_{osc,x}^B$ and $\mathring{R}_{osc,x}^u$, since $W_{(k)}$ is $(\mathbb{T}/\sigma)^2$-periodic, we get
$$\mathbb{P}_{\neq0}(W_{(k)}\otimes W_{(k)})=\mathbb{P}_{\geq\sigma/2}(W_{(k)}\otimes W_{(k)}).$$
Then, by estimates \eqref{3.14}, \eqref{4.7}, \eqref{5.14'}, and \eqref{1.10}, we apply Lemma \ref{lem5.1} with $a=\nabla(a_{(k)}^2)$ and $f=\phi_{(k)}^2$ to obtain
\begin{align}\label{5.26}
\|\mathring{R}_{osc,x}^B\|_{L^p}&\lesssim\sum_{k\in \Lambda_2}\||\nabla|^{-1}\mathbb{P}_{\neq0}\left(\nabla(a_{(k)}^2)\mathbb{P}_{\geq\sigma/2}(W_{(k)}\otimes W_{(k)})\right)\|_{L^p}\notag\\
&\lesssim\sum_{k\in \Lambda_2}\sigma^{-1}\|\nabla(a_{(k)}^2)\|_{C^2_{t,x}}\|\phi_{(k)}^2\|_{L^p}\notag\\
&\lesssim\sigma^{-1}l^{-13}\mu^{1-\frac{1}{p}}\notag\\
&\lesssim l^{-13}\lambda_{q+1}^{-\varepsilon+\frac{\varepsilon}{8}}\lesssim \lambda_{q+1}^{-\frac{\varepsilon}{2}}.
\end{align}
By a similar argument, one has
\begin{align}\label{5.27}
\|\mathring{R}_{osc,x}^u\|_{L^p}\lesssim\lambda_{q+1}^{-\frac{\varepsilon}{2}}.
\end{align}

For the term $\mathring{R}_{osc,far}^B$, we employ estimates \eqref{3.16} and \eqref{5.17}, the Calderon-Zygmund inequality, Lemma \ref{lem4.1} and Lemma \ref{lem4.2} to obtain
\begin{align}\label{5.28}
\|\mathring{R}_{osc,far}^B\|_{L^p}&\lesssim\sum_{k\neq k'\in \Lambda_2}\|a_{(k)}a_{(k')}\|_{L^\infty}\|W_{(k)}\otimes W_{(k')}\|_{L^p}\notag\\
&\quad+\sum_{k\in \Lambda_2, k'\in \Lambda_1}\||\nabla|^{-1}\mathbb{P}_{\neq0}\left(a_{(k)}a_{(k')}(W_{(k)}\otimes W_{(k')}-W_{(k')}\otimes W_{(k)})\right)\|_{L^p}\notag\\
&\lesssim l^{-2}\mu^{1-\frac{2}{p}}\lesssim l^{-2}\lambda_{q+1}^{\frac{\varepsilon}{8}-\frac{1}{2}}\lesssim\lambda_{q+1}^{-\frac{\varepsilon}{4}},
\end{align}
where the last two inequalities is obtained by \eqref{5.14'}, \eqref{1.9} and \eqref{1.10}.
Similarly, one has
\begin{align}\label{5.29}
\|\mathring{R}_{osc,far}^u\|_{L^p}\lesssim\lambda_{q+1}^{-\frac{\varepsilon}{4}}.
\end{align}
Then, inserting estimates \eqref{5.26}--\eqref{5.29} into \eqref{5.24} and \eqref{5.25}, we have
\begin{align}\label{5.30}
\|(\mathring{R}_{osc}^u,\mathring{R}_{osc}^B)\|_{L^p}\lesssim\lambda_{q+1}^{-\frac{\varepsilon}{4}}.
\end{align}

\underline{Corrector errors.}
Recalling the definitions of the corrector errors $\mathring{R}_{cor}^B$ and $\mathring{R}_{cor}^u$ in \eqref{5.3} and \eqref{5.9}, respectively; by utilizing the Calderon-Zygmund inequality, estimates \eqref{5.17}, \eqref{4.30'}, \eqref{4.31''} and \eqref{4.31'}, we have
\begin{align}\label{5.31}
\|\mathring{R}_{cor}^B\|_{L^p}&\lesssim\||\nabla|^{-1}\mathbb{P}_{\neq0}\left(d_{q+1}^{(p)}\otimes w_{q+1}^{(c)}-w_{q+1}^{(c)}\otimes d_{q+1}^{(p)}+d_{q+1}^{(c)}\otimes w_{q+1}-w_{q+1}\otimes d_{q+1}^{(c)}\right)\|_{L^p}\notag\\
&\quad+\|d_{q+1}^{(p)}\otimes d_{q+1}^{(c)}+d_{q+1}^{(c)}\otimes d_{q+1}\|_{L^p}\notag\\
&\lesssim\|(w_{q+1}^{(c)},d_{q+1}^{(c)})\|_{L^{2p}}(\|d_{q+1}\|_{L^{2p}}+\|d_{q+1}^{(p)}\|_{L^{2p}}+\|w_{q+1}\|_{L^{2p}})\notag\\
&\lesssim \lambda_{q+1}^{-\frac{1}{2}}\lambda_{q+1}^{\frac{1}{2}-\frac{1}{2p}}\lesssim \lambda_{q+1}^{-\frac{\varepsilon}{4}},
\end{align}
where the last inequality is obtained by \eqref{5.14'}. Similarly, we have
\begin{align}\label{5.32}
\|\mathring{R}_{cor}^u\|_{L^p}\lesssim \lambda_{q+1}^{-\frac{\varepsilon}{4}}.
\end{align}
Then from \eqref{5.31} and \eqref{5.32}, we get
\begin{align}\label{5.33}
\|(\mathring{R}_{cor}^u,\mathring{R}_{cor}^B)\|_{L^p}\lesssim \lambda_{q+1}^{-\frac{\varepsilon}{4}}.
\end{align}

\underline{Commutator errors.}
By the Calderon-Zygmund inequality, estimates \eqref{2.10}, \eqref{2.11} and \eqref{5.17}, we infer
\begin{align*}
&\|\mathring{R}_{com}^B\|_{L^p}\lesssim\|\mathring{R}_{com,1}^B\|_{L^p}+\||\nabla|^{-1}\mathbb{P}_{\neq0}\mathring{R}_{com,2}^B\|_{L^p}\lesssim \lambda_{q}^{-6},\\
&\|\mathring{R}_{com'}^u\|_{L^p}\lesssim\|\mathring{R}_{com}^u\|_{L^p}\lesssim \lambda_{q}^{-6}.
\end{align*}
Thus,
\begin{align}\label{5.34}
\|(\mathring{R}_{com'}^u,\mathring{R}_{com}^B)\|_{L^p}\lesssim \lambda_{q}^{-6}.
\end{align}

\subsection{Proof of the main iteration in Theorem \ref{main iteration}}
The inductive estimates \eqref{1.12}, \eqref{1.15} and \eqref{1.16} are justified in Section \ref{sec4.3}. Now we prove estimate \eqref{1.14}. Inserting estimates \eqref{5.23}, \eqref{5.30}, \eqref{5.33} and \eqref{5.34} into \eqref{5.2'} and \eqref{5.7}, we have
\begin{align}\label{5.35}
\|(\mathring{R}_{q+1}^u,\mathring{R}_{q+1}^B)\|_{L^1}\lesssim \|(\mathring{R}_{q+1}^u,\mathring{R}_{q+1}^B)\|_{L^p}\lesssim \lambda_{q+1}^{-\frac{\varepsilon}{4}}+\lambda_{q}^{-6}\leq\delta_{q+2},
\end{align}
where the last inequality is given by the fact that $\frac{\varepsilon}{b}\gg\beta$ and $\beta b^2<\frac{1}{100}$, which are deduced from \eqref{1.9} and \eqref{1.10}.

Next, we prove inductive estimate \eqref{1.13}. By Sobolev's embedding $W^{1,3}(\mathbb{T}^2)\hookrightarrow L^\infty(\mathbb{T}^2)$, equation \eqref{5.13} and estimates \eqref{2.7}, \eqref{4.30}, we find that
\begin{align*}
\|\mathring{R}_{q+1}^u\|_{C_tC^1_x}&\lesssim\|\mathcal{R}\mathbb{P}_{H}{\rm div}\mathring{R}_{q+1}^u\|_{C_tW^{2,3}_x}\\
&\lesssim\|\partial_tu_{q+1}+\nu_1(-\Delta)^{\alpha_1} u_{q+1}+{\rm div}(u_{q+1}\otimes u_{q+1}-B_{q+1}\otimes B_{q+1})\|_{C_tW^{1,3}_x}\\
&\lesssim\|u_{q+1}\|_{C^1_tW_x^{1,3}}+\| u_{q+1}\|_{C_tW_x^{2\alpha_1+1,3}}+\sum_{0\leq N'\leq2}\|B_{q+1}\|_{C^{N'}_{t,x}}\|B_{q+1}\|_{C^{2-N'}_{t,x}}\\
&\quad+\sum_{0\leq N'\leq2}\|u_{q+1}\|_{C^{N'}_{t,x}}\|u_{q+1}\|_{C^{2-N'}_{t,x}}\\
&\lesssim \lambda_{q+1}^5+\lambda_{q+1}^7\leq\lambda_{q+1}^8.
\end{align*}
By the same reason, we have
\begin{align*}
&\|\mathring{R}_{q+1}^B\|_{C_tC^1_x}\\
&\lesssim\|\mathcal{R}{\rm curl}^{-1}(\partial_tB_{q+1}+\nu_2(-\Delta)^{\alpha_2} B_{q+1}+{\rm div}(B_{q+1}\otimes u_{q+1}-u_{q+1}\otimes B_{q+1})\\
&\quad+\nabla\times{\rm div}(B_{q+1}\otimes B_{q+1}))\|_{C_tW^{2,3}_x}\\
&\lesssim\|B_{q+1}\|_{C^1_{t,x}}+\|B_{q+1}\|_{C_tW_x^{2\alpha_2,3}}+\sum_{0\leq N'\leq1}\|u_{q+1}\|_{C^{N'}_{t,x}}\|B_{q+1}\|_{C^{1-N'}_{t,x}}\\
&\quad+\sum_{0\leq N'\leq2}\|B_{q+1}\|_{C^{N'}_{t,x}}\|B_{q+1}\|_{C^{2-N'}_{t,x}}\\
&\lesssim\lambda_{q+1}^8.
\end{align*}
Similarly,
\begin{align*}
\|\partial_t\mathring{R}_{q+1}^u\|_{C_{t,x}}&\lesssim\|\partial^2_tu_{q+1}+\nu_1(-\Delta)^{\alpha_1} \partial_tu_{q+1}+{\rm div}\partial_t(u_{q+1}\otimes u_{q+1}-B_{q+1}\otimes B_{q+1})\|_{C_tL^{3}_x}\\
&\lesssim\|u_{q+1}\|_{C^2_{t,x}}+\|u_{q+1}\|_{C^3_{t,x}}+\|u_{q+1}\otimes u_{q+1}-B_{q+1}\otimes B_{q+1}\|_{C^2_{t,x}}\\
&\lesssim \lambda_{q+1}^8.
\end{align*}
By Sobolev's embedding $H^{2}(\mathbb{T}^2)\hookrightarrow L^\infty(\mathbb{T}^2)$, and estimates \eqref{2.7}, \eqref{4.30}, we obtain
\begin{align*}
&\|\partial_t\mathring{R}_{q+1}^B\|_{C_{t,x}}\\
&\lesssim\|\partial_tB_{q+1}+\nu_2(-\Delta)^{\alpha_2} B_{q+1}+{\rm div}(B_{q+1}\otimes u_{q+1}-u_{q+1}\otimes B_{q+1})\\
&\quad+\nabla\times{\rm div}(B_{q+1}\otimes B_{q+1})\|_{C^1_tL^2_x}\\
&\lesssim\|B_{q+1}\|_{C^2_{t,x}}+\|B_{q+1}\|_{C^1_tW_x^{2\alpha_2,3}}\\
&\quad+\sum_{0\leq N'\leq2}\|u_{q+1}\|_{C^{N'}_{t,x}}\|B_{q+1}\|_{C^{2-N'}_{t,x}}+\sum_{0\leq N'\leq3}\|B_{q+1}\|_{C^{N'}_{t,x}}\|B_{q+1}\|_{C^{3-N'}_{t,x}}\\
&\lesssim\lambda_{q+1}^8.
\end{align*}
Thus \eqref{1.13} is verified.

Now we justify \eqref{1.17}. By the definitions of $\mathring{R}_{q+1}^u$ in \eqref{5.7} and $\mathring{R}_{q+1}^B$ in \eqref{5.2'}, we have
\begin{align}\label{5.36}
{\rm supp}_t(\mathring{R}_{q+1}^u,\mathring{R}_{q+1}^B)\subset\bigcup_{k \in \Lambda_1\cup\Lambda_2}{\rm supp}_ta_{(k)}\cup{\rm supp}_t(u_l,B_l)\cup{\rm supp}_t(\mathring{R}_{q}^u,\mathring{R}_{q}^B),
\end{align}
and by the definitions of $u_{q+1}$ in \eqref{4.24} and $B_{q+1}$ in \eqref{4.12}, one has
\begin{align}\label{5.37}
{\rm supp}_t(u_{q+1},B_{q+1})\subset{\rm supp}_t(u_l,B_l)\cup{\rm supp}_t(w_{q+1},d_{q+1}).
\end{align}
From the definitions and \eqref{2.12} and \eqref{2.13}, we can infer that
\begin{align}\label{5.38}
{\rm supp}_t(w_{q+1},d_{q+1})\subset\bigcup_{k \in \Lambda_1\cup\Lambda_2}{\rm supp}_ta_{(k)}\subset N_{3l}\left({\rm supp}_t(\mathring{R}_{q}^u,\mathring{R}_{q}^B)\right).
\end{align}
Thus, \eqref{5.36}--\eqref{5.38} yields that
\begin{align}\label{5.39}
{\rm supp}_t(u_{q+1},B_{q+1},\mathring{R}_{q+1}^u,\mathring{R}_{q+1}^B)\subset N_{3l}\left({\rm supp}_t(u_q,B_q,\mathring{R}_{q}^u,\mathring{R}_{q}^B)\right),
\end{align}
and since $l\ll\delta_{q+2}^{\frac{1}{2}}$, \eqref{1.17} is verified.

\section{Proof of main results}\label{sec6}
\subsection{Proof of Theorem \ref{main result}}
Let $u_0=\tilde{u}$, $B_0=\tilde{B}$, and
\begin{align}
&\mathring{R}_0^u=\mathcal{R}(\partial_tu_0+\nu_1(-\Delta)^{\alpha_1} u_{0})+u_{0}\mathring{\otimes} u_{0}-B_{0}\mathring{\otimes} B_{0},\label{6.1}\\
&\mathring{R}_0^B=\mathcal{R}{\rm curl}^{-1}\left(\partial_tB_0+\nu_2(-\Delta)^{\alpha_2} B_{0}+{\rm div}(B_{0}\otimes u_{0}-u_{0}\otimes B_{0})\right)+B_{0}\mathring{\otimes} B_{0},\label{6.2}\\
&p_0=\frac{1}{3}(|u_0|^2-|B_0|^2),\label{6.3}
\end{align}
then $(u_0,B_0,\mathring{R}_0^u,\mathring{R}_0^B)$ is a solution of \eqref{Hall-MHDapp} with pressure $p_0$.

Set $\delta_1=\max\{\|\mathring{R}_0^u\|_{L^\infty_tL^1_x},\|\mathring{R}_0^B\|_{L^\infty_tL^1_x}\}$ and $a$, $M$ large enough such that $(u_0,B_0,\mathring{R}_0^u,\mathring{R}_0^B)$ satisfies \eqref{1.12}--\eqref{1.14} at level $q=0$. Therefore, by Theorem \ref{main iteration}, there exists a sequence of solutions $(u_{q+1},B_{q+1},\mathring{R}_{q+1}^u,\mathring{R}_{q+1}^B)$ to \eqref{Hall-MHDapp} with \eqref{1.12}--\eqref{1.14} holds for $q\geq0$. Then by the interpolation inequality, \eqref{1.12} and \eqref{1.15}, for any $\beta'\in (0,\frac{\beta}{4+\beta})$, we have{\small
\begin{align}\label{6.4}
\sum_{q\geq0}\|u_{q+1}-u_{q}\|_{\dot{H}^{\beta'}}\lesssim\sum_{q\geq0}\|u_{q+1}-u_{q}\|_{L^2}^{1-\beta'}\|u_{q+1}-u_{q}\|_{H^1}^{\beta'}
\lesssim\sum_{q\geq0}M^{1-\beta'}\delta_{q+1}^{\frac{1}{2}(1-\beta')}\lambda_{q+1}^{4\beta'}\lesssim1.
\end{align}}
Similarly,
\begin{align*}
\sum_{q\geq0}\|B_{q+1}-B_{q}\|_{\dot{H}^{\beta'}}\lesssim1.
\end{align*}
Thus we define $(u,B)=\lim_{q\rightarrow\infty}(u_q,B_q)$ in $H^{\beta'}$. In view of \eqref{1.14}, we infer that $(u,B)$ is a weak solution to Hall-MHD equation \eqref{Hall-MHD1} in the sense of Definition \ref{def1}.

Next, we prove inequality \eqref{thm1.2-2}. From \eqref{1.16} in Theorem \ref{main iteration}, we have
\begin{align*}
\|u-\tilde{u}\|_{L^1_x}\leq\sum_{q\geq0}\|u_{q+1}-u_{q}\|_{L_x^{1}}\leq\sum_{q\geq0}\delta_{q+2}^{\frac{1}{2}}
\leq\sum_{q\geq0}\lambda_{q+2}^{-\beta}
\leq\sum_{q\geq2}a^{-\beta bq}
=\frac{a^{-2\beta b}}{1-a^{-\beta b}}\leq\varepsilon_*,
\end{align*}
for $a$ large enough (depending on $\varepsilon_*$). By a similar argument, one has $\|B-\tilde{B}\|_{L^1_x}\leq\varepsilon_*$, then \eqref{thm1.2-2} is verified.

Since
$${\rm supp}_t(\mathring{R}^u_0,\mathring{R}^B_0)\subset{\rm supp}_t(\tilde{u},\tilde{B}),$$
by \eqref{1.17}, we conclude that
\begin{align*}
{\rm supp}_t(u,B)&\subset\bigcap_{q\geq0}{\rm supp}_t(u_q,B_q,\mathring{R}^u_q,\mathring{R}^B_q)=N_{\Sigma_{q\geq0}}\left({\rm supp}_t(u_0,B_0,\mathring{R}^u_0,\mathring{R}^B_0)\right)\\
&\subset N_{\varepsilon_*}\left({\rm supp}_t(\tilde{u},\tilde{B})\right),
\end{align*}
then \eqref{thm1.2-3} is verified.

Next we prove \eqref{thm1.2-4}. Assuming that the mean-free periodic vectors $A$ and $A_0$ are the potential fields corresponding to $B$ and $B_0$, respectively, we have
\begin{align}\label{6.10}
|\mathcal{H}_{B,B}-\mathcal{H}_{\tilde{B},\tilde{B}}|&=\left|\int_{\mathbb{T}^2}A\cdot B-A_0\cdot B_0dx\right|\notag\\
&\leq\left|\int_{\mathbb{T}^2}A\cdot( B-B_0)dx\right|+\left|\int_{\mathbb{T}^2}(A-A_0)\cdot B_0 dx\right|\notag\\
&\leq\|A\|_{L^{\infty}}\|B-B_0\|_{L^{1}}+\|B_0\|_{L^{\infty}}\|A-A_0\|_{L^{1}}.
\end{align}
Noting that $B$ is divergence free, by the Biot-Savart law (see \cite{kryz}, for example), we know that $\nabla\times (-\Delta)^{-1}B=A$. By Sobolev's embedding, one has
$$\|A\|_{{L^\infty}(\mathbb{T}^2)}\lesssim\|A\|_{\dot{H}^{1+\beta'}(\mathbb{T}^2)}\lesssim\|B\|_{H^{\beta'}(\mathbb{T}^2)}\lesssim 1.$$
Noting that $A$ and $A_0$ are mean-free, by inequality \eqref{5.17}, \eqref{6.10} can be estimated as
\begin{align*}
|\mathcal{H}_{B,B}-\mathcal{H}_{\tilde{B},\tilde{B}}|\lesssim \|B-B_0\|_{L^{1}}\lesssim\sum_{q\geq0}\|B_{q+1}-B_{q}\|_{L^{1}}\lesssim \sum_{q\geq0} \delta_{q+2}^{\frac{1}{2}}\leq\varepsilon_*.
\end{align*}

\subsection{Proof of Theorem \ref{main theorem}}
For $m\in \mathbb{N}_+$, we define the incompressible, mean-free fields $\tilde{u}_m$ and $\tilde{B}_m$ in $C^\infty([0,1]\times\mathbb{T}^2;\mathbb{R}^3)$ by
\begin{align}
&\tilde{u}_m:=m\psi(t)(\sin x_2,0,0)^T,\label{6.5}\\
&\tilde{B}_m:=m\psi(t)(\sin x_2,\cos x_1,-\sin x_1-\cos x_2)^T,\label{6.6}
\end{align}
where $\psi(t):[0,1]\rightarrow\mathbb{R}$, is any cut-off function supported on $[\frac{1}{4},\frac{3}{4}]$, and $\psi(t)=1$ on $[\frac{1}{2},\frac{5}{8}]$, $0\leq\psi\leq1$. For any $t\in[\frac{1}{2},\frac{5}{8}]$, one has
\begin{align}\label{6.7}
\|\tilde{u}_m\|_{L^1_x(\mathbb{T}^2)}=8\pi m,\quad\|\tilde{B}_m\|_{L^1_x(\mathbb{T}^2)}=16\pi m+m\int_{\mathbb{T}^2}|\sin x_1+\cos x_2|dx.
\end{align}

Using Theorem \ref{main result} with $\varepsilon_*=\frac{1}{100}$, there exists $(u_m,B_m)\in H^{\beta'}_x$ such that $(u_m,B_m)$ solves equation \eqref{Hall-MHD1} and
\begin{align}\label{6.8}
\|u_m-\tilde{u}_m\|_{L^1_x}\leq\varepsilon_*,\quad\|B_m-\tilde{B}_m\|_{L^1_x}\leq\varepsilon_*,\quad|\mathcal{H}_{B_m,B_m}-\mathcal{H}_{\tilde{B}_m,\tilde{B}_m}|\leq\varepsilon_*.
\end{align}
Then for any $m>m'$ and $t\in[\frac{1}{2},\frac{5}{8}]$, by \eqref{6.7} and \eqref{6.8}, we can infer that
\begin{align*}
|\|u_m\|_{L^1_x}-\|u_{m'}\|_{L^1_x}|&=|\|u_m\|_{L^1_x}-\|\tilde{u}_m\|_{L^1_x}+\|\tilde{u}_m\|_{L^1_x}-\|\tilde{u}_{m'}\|_{L^1_x}+\|\tilde{u}_{m'}\|_{L^1_x}-\|u_{m'}\|_{L^1_x}|\\
&\geq|\|\tilde{u}_m\|_{L^1_x}-\|\tilde{u}_{m'}\|_{L^1_x}|-|\|u_m\|_{L^1_x}-\|\tilde{u}_m\|_{L^1_x}|-|\|\tilde{u}_{m'}\|_{L^1_x}-\|u_{m'}\|_{L^1_x}|\\
&\geq8\pi(m-m')-2\varepsilon_*>0.
\end{align*}
Similarly, one has
\begin{align*}
|\|B_m\|_{L^1_x}-\|B_{m'}\|_{L^1_x}|>0,
\end{align*}
for $t\in[\frac{1}{2},\frac{5}{8}]$. Hence, we conclude that
\begin{align*}
(u_{m},B_{m})\neq(u_{m'},B_{m'})\;\; {\rm for}\;\;m\neq m'.
\end{align*}
Notice that ${\rm supp}_t \phi\subset[\frac{1}{4},\frac{3}{4}]$, then one has
\begin{align}\label{6.9}
u_m(0)=B_m(0)=0,
\end{align}
which implies that there are infinitely many different weak solutions of system \eqref{Hall-MHD1} with the same initial data.

Next, we prove that $(u_m,B_m)$ do not conserve the magnetic helicity. Let
$$\tilde{A}_m=m\psi(t)(\sin x_2,\cos x_1,-\sin x_1-\cos x_2)^T$$
be the potential field of $\tilde{B}_m$, that is $\nabla\times\tilde{A}_m=\tilde{B}_m$. Then one has
$$\mathcal{H}_{\tilde{B}_m,\tilde{B}_m}=8\pi^2m^2\psi(t)^2.$$
Thus for any $t\in [\frac{1}{2},\frac{5}{8}]$ and $m\in \mathbb{N}_+$, by \eqref{6.8}, we have
$$\mathcal{H}_{B_m,B_m}\geq8\pi^2m^2-\varepsilon_*>4\pi^2.$$
On the other hand, \eqref{6.9} implies that $\mathcal{H}_{B_m,B_m}(0)=0$. Therefore, we conclude that the magnetic helicity $\mathcal{H}_{B_m,B_m}$ is not conserve.

\subsection{Proof of Theorem \ref{thm1.4}}
For fixed $\bar{\beta}>0$, assume that a weak solution of the Hall-MHD equations $(u,B)$ belongs to $C_{t,x}^{\bar{\beta}}\times C_{t,x}^{\bar{\beta}}$. Let $\phi_\epsilon$, $\varphi_\epsilon$ be the standard 2D and 1D Friedrichs mollifiers. Set
\begin{align}\label{6.11}
u_n=\left(u\ast_x\phi_{\lambda_n^{-1}}\right)\ast_t\varphi_{\lambda_n^{-1}}, \quad B_n=\left(B\ast_x\phi_{\lambda_n^{-1}}\right)\ast_t\varphi_{\lambda_n^{-1}}.
\end{align}
Then $(u_n,B_n)$ solves {\small
\begin{equation}\label{Hall-MHDn}
\begin{cases}
\partial _tu_n+\nu_{1,n}(-\Delta)^{\alpha_1} u_n+{\rm div}(u_n\otimes u_n-B_n\otimes B_n)+\nabla p_n={\rm div}\mathring{R}_n^u, \\
\partial _tB_n+\nu_{2,n}(-\Delta)^{\alpha_2} B_n+{\rm div}(B_n\otimes u_n-u_n\otimes B_n)+\nabla\times{\rm div}\left( B_n\otimes B_n\right)=\nabla\times{\rm div}\mathring{R}_n^B,\\
{\rm div}u_n ={\rm div}B_n=0,
\end{cases}
\end{equation}}
\!\!where
$$\nu_{1,n}:=\lambda_n^{-2\alpha_1},\quad \nu_{2,n}:=\lambda_n^{-2\alpha_2},$$
and the symmetric traceless stresses $\mathring{R}_n^u$ and $\mathring{R}_n^B$ are given by
\begin{align}
&\mathring{R}_n^u=u_n\mathring{\otimes} u_n-B_n\mathring{\otimes} B_n-(u\mathring{\otimes} u-B\mathring{\otimes} B)\ast\phi_{\lambda_n^{-1}}\ast\varphi_{\lambda_n^{-1}}+\nu_{1,n}\mathcal{R}(-\Delta)^{\alpha_1} u_n,\label{6.13}\\
&\mathring{R}_n^B=B_n\mathring{\otimes} B_n-(B\mathring{\otimes} B)\ast\phi_{\lambda_n^{-1}}\ast\varphi_{\lambda_n^{-1}}+\nu_{2,n}\mathcal{R}{\rm curl}^{-1}(-\Delta)^{\alpha_2} B_n,\notag\\
&\quad\quad+\mathcal{R}{\rm curl}^{-1}{\rm div}\left(B_n\otimes u_n-u_n\otimes B_n-\left(B\otimes u-u\otimes B\right)\ast\phi_{\lambda_n^{-1}}\ast\varphi_{\lambda_n^{-1}}\right)\label{6.14},
\end{align}
and
\begin{align}\label{6.13}
p_n=p\ast\phi_{\lambda_n^{-1}}\ast\varphi_{\lambda_n^{-1}}+\frac{1}{3}(|u|^2-|B|^2)\ast\phi_{\lambda_n^{-1}}\ast\varphi_{\lambda_n^{-1}}.
\end{align}
By the commutator estimate \cite[Lemma 2.1]{cds}, the Calderon-Zygmund inequality, one has
\begin{align}\label{6.16}
\|\mathring{R}_n^u\|_{L^1_x}&\lesssim\|u_n\mathring{\otimes} u_n-B_n\mathring{\otimes} B_n-(u\mathring{\otimes}u-B\mathring{\otimes} B)\ast\phi_{\lambda_n^{-1}}\ast\varphi_{\lambda_n^{-1}}\|_{C^0_x}+\lambda_n^{-2\alpha_1}\|\mathcal{R}(-\Delta)^{\alpha_1} u_n\|_{L^2_x}\notag\\
&\lesssim\lambda_n^{-2\bar{\beta}}(\|u\|^2_{C_{t,x}^{\bar{\beta}}}+\|B\|^2_{C_{t,x}^{\bar{\beta}}})+\lambda_n^{-2\alpha_1}\lambda_n^{2\alpha_1-1}\|u\|_{L^2_x}\notag\\
&\lesssim\lambda_n^{-2\bar{\beta}}+\lambda_n^{-1}.
\end{align}
Similarly,
\begin{align}\label{6.17}
\|\mathring{R}_n^B\|_{L^1_x}&\lesssim\lambda_n^{-2\bar{\beta}}(\|u\|^2_{C_{t,x}^{\bar{\beta}}}+\|B\|^2_{C_{t,x}^{\bar{\beta}}})+\lambda_n^{-2\alpha_2}\lambda_n^{2\alpha_2-2}\|B\|_{L^2_x}\notag\\
&\lesssim\lambda_n^{-2\bar{\beta}}+\lambda_n^{-2}.
\end{align}
A similar argument implies that
\begin{align}
\|(\mathring{R}_n^u,\mathring{R}_n^B)\|_{C^1_{t,x}}&\lesssim\lambda_n^{-(2\bar{\beta}-1)}(\|u\|^2_{C_{t,x}^{\bar{\beta}}}+\|B\|^2_{C_{t,x}^{\bar{\beta}}})+\lambda_n^{-2\alpha_1}\lambda_n^{2\alpha_1}\|u\|_{C_{t,x}}+\lambda_n^{-2\alpha_2}\lambda_n^{2\alpha_2-1}\|B\|_{C_{t,x}}\notag\\
&\lesssim\lambda_n^{1-2\bar{\beta}}+1,\label{6.18}\\
\|(u_n,B_n)\|_{C^1_{t,x}}&\lesssim\lambda_n^{1-\bar{\beta}}\|(u,B)\|_{C^{\bar{\beta}}_{t,x}}\lesssim\lambda_n^{1-\bar{\beta}}.\label{6.20}
\end{align}
Thus, $(u_n,B_n,\mathring{R}_n^u,\mathring{R}_n^B)$ satisfies the inductive estimates \eqref{1.12}--\eqref{1.14} at level $q=n(\geq1)$ provided that $a$ is sufficiently large. Hence, we apply Theorem \ref{main iteration} to obtain a sequence of solutions $(u_{n,q},B_{n,q},\mathring{R}_{n,q}^u,\mathring{R}_{n,q}^B)_{q\geq n}$ satisfying \eqref{1.12}--\eqref{1.14}.

Setting $q\rightarrow\infty$, by the proof of Theorem \ref{main result}, we can obtain that a weak solution to the Hall-MHD system \eqref{Hall-MHD1} $(u^{\nu_{1,n}},B^{\nu_{2,n}})\in C^0_{t}H_x^{\beta'}\times C^0_{t}H_x^{\beta'}$  for some $\beta'\in (0,\frac{\beta}{4+\beta})$. Moreover, as in \eqref{6.4}, we get
\begin{align}\label{6.21}
\|u^{\nu_{1,n}}-u_n\|_{C^0_{t}H_x^{\beta'}}+\|B^{\nu_{2,n}}-B_n\|_{C^0_{t}H_x^{\beta'}}\leq\sum_{q=n}^{\infty}\lambda_{q+1}^{4\beta'}\lambda_{q+1}^{-\beta(1-\beta')}\leq\frac{1}{2n}.
\end{align}
Using the commutator estimate \cite[Lemma 2.1]{cds} again, by \eqref{6.21}, we have
\begin{align}\label{6.22}
\|u^{\nu_{1,n}}-u\|_{H_x^{\beta'}}\leq \|u^{\nu_{1,n}}-u_n\|_{H_x^{\beta'}}+\|u_n-u\|_{H_x^{\beta'}}\leq\frac{1}{2n}+C\lambda_n^{-(2\bar{\beta}-\beta')}\|u\|_{C^{\bar{\beta}}_{t,x}}\leq\frac{1}{n}
\end{align}
provided that $0<\beta'<\min\left\{\bar{\beta},\frac{\beta}{4+\beta}\right\}$. A similar argument implies that $\|B^{\nu_{2,n}}-B\|_{H_x^{\beta'}}\leq\frac{1}{n}$. Thus, the proof of Theorem \ref{thm1.4} is completed.


\section*{appendix}\label{sec7}
\begin{proof}[Proof of Lemma \ref{lem3.1}]
Let $e_1=(1,0,0)$, $e_2=(0,1,0)$, $e_3=(0,0,1)$, and let $\Lambda_1=\{\frac{3}{5}e_1\pm \frac{4}{5}e_2,\frac{4}{5}e_1\pm \frac{3}{5}e_3,\frac{3}{5}e_2\pm \frac{4}{5}e_3\}$. The orthonormal bases are given by
\begin{table}[H]
\begin{tabular}{c|c|c}
  $k$ & $A_k$  & $k\times A_k$  \\\hline
  $\frac{3}{5}e_1\pm \frac{4}{5}e_2$& $\frac{4}{5}e_1\mp \frac{3}{5}e_2$ & $e_3$\\
  $\frac{4}{5}e_1\pm \frac{3}{5}e_3$& $e_2$ & $\frac{3}{5}e_1\mp \frac{4}{5}e_3$\\
  $\frac{3}{5}e_2\pm \frac{4}{5}e_3$& $e_1$& $\frac{4}{5}e_2\mp \frac{3}{5}e_3$\\
\end{tabular}
\caption{}
\end{table}\label{tab1}\!\!\!\!\!\!\!\!\!\!
\noindent Then, we have $\sum_{k \in\Lambda_1}\frac{1}{2}k\otimes k={\rm Id}$. The implicit function theorem implies that there exists a small constant $\delta_0$ such that for all symmetric matrices $R$ with $|R-{\rm Id}|<\delta_0$, it holds
$$R=\sum_{k\in \Lambda_1}\gamma^2_{k}(R)(k\otimes k),$$
for some smooth functions $\gamma_k$ with $\gamma^2_{k}({\rm Id})=\frac{1}{2}$. For more details, see \cite{bdis}.

Similarly, take $\Lambda_2=\{\frac{5}{13}e_1\pm \frac{12}{13}e_2,\frac{12}{13}e_1\pm \frac{5}{13}e_3,\frac{5}{13}e_2\pm \frac{12}{13}e_3\}$, with the orthonormal bases are given by
\begin{table}[H]
\begin{tabular}{c|c|c}
  $k$ & $A_k$  & $k\times A_k$  \\\hline
  $\frac{5}{13}e_1\pm \frac{12}{13}e_2$& $\frac{12}{13}e_1\mp \frac{5}{13}e_2$ & $e_3$\\
  $\frac{12}{13}e_1\pm \frac{5}{13}e_3$& $e_2$ & $\frac{5}{13}e_1\mp \frac{12}{13}e_3$\\
  $\frac{5}{13}e_2\pm \frac{12}{13}e_3$& $e_1$& $\frac{12}{13}e_2\mp \frac{5}{13}e_3$
\end{tabular}
\end{table}\label{tab2}
\noindent Repeating the above process, we conclude that the same conclusion holds for $\Lambda_2$. Moreover, one has $\Lambda_{1}\cap\Lambda_{2}=\emptyset$.
\end{proof}
\underline{The definitions of $\tilde{k}$ and $\tilde{A_k}$.} In order to build the $2\frac{1}{2}$ dimensional intermittent flows, we define $\tilde{k}, \;\tilde{A_k}\in \mathbb{R}^2$ by the orthonormal base of the plane $\{k\times A_k=0\}$. More precisely, for $k\in \Lambda_1$, we define
\begin{table}[H]
\begin{tabular}{c|c|c}
  $k$ & $\tilde{k}$  & $\tilde{A_k}$  \\\hline
  $\frac{3}{5}e_1\pm \frac{4}{5}e_2$& $(\frac{3}{5},\pm\frac{4}{5})$ & $(\frac{4}{5},\mp\frac{3}{5})$\\
  $\frac{4}{5}e_1\pm \frac{3}{5}e_3$& $(1,0)$ & $(0,1)$\\
  $\frac{3}{5}e_2\pm \frac{4}{5}e_3$& $(0,1)$& $(1,0)$\\
\end{tabular}
\caption{}\label{tab3}
\end{table}
\noindent For $k\in \Lambda_2$, $\tilde{k}$ and $\tilde{A_k}$ are defined in a similar way. Moreover, one has $k\cdot (\tilde{A_k},0)=0$.

\section*{Acknowledgments}
Y. Peng is supported by the scientific research starting project of SWPU (Grant No. 2024QHZ018). H. Wang is supported by the National Natural Science Foundation of China (Grant No.12471208), the Natural Science Foundation of Chongqing (Grant No. CSTB2023NSCQ-MSX0396).

\noindent {\bf Statements and Declarations}

\noindent Conflicts of interest/Competing interests: The authors declare that they have no
competing interests.

\noindent {\bf Data availability}

\noindent Data sharing not applicable to this article as no datasets were generated or analysed during the current study.


\begin{thebibliography}{aa}
\footnotesize

\bibitem{adfl}
M. Acheritogaray, P. Degond, A. Frouvelle, J.G. Liu, Kinetic formulation and global existence for the Hall-magnetohydrodynamic system, Kinet. Relat. Models, 4 (2011) 901--918.

\bibitem{ABC}
D. Albritton, E. Br\'{u}e,  M. Colombo. Non-uniqueness of leray solutions of the forced Navier--Stokes
equations, Ann. Math., 196 (2022) 415-455.

\bibitem{bbv}
R. Beekie, T. Buckmaster, V. Vicol, Weak solutions of ideal MHD which do not conserve magnetic helicity, Ann. PDE, 6 (2020) 1--40.

\bibitem{bcv}
T. Buckmaster, M. Colombo, V. Vicol, Wild solutions of the Navier--Stokes equations whose singular sets in time have Hausdorff dimension strictly less than 1, J. Eur. Math. Soc., 24(9) (2021) 3333--3378.

\bibitem{bdis}
T. Buckmaster, C. De Lellis, P. Isett, L. Sz\'{e}kelyhidi Jr, Anomalous dissipation for 1/5-H\"{o}lder Euler flows, Ann. Math., 182(1) (2015) 127--172.

\bibitem{bt}
S.A. Balbus, C. Terquem, Linear analysis of the Hall effect in protostellar disks, Astrophys. J., 552 (2001) 235--247.

\bibitem{bv}
T. Buckmaster, V. Vicol, Nonuniqueness of weak solutions to the Navier--Stokes equation, Ann. Math., 189(2) (2019) 101--114.

\bibitem{cdl}
D. Chae, P. Degond, J.G. Liu, Well-posedness for Hall-magnetohydrodynamics, Ann. Inst. H. Poincar\'{e} Anal. Non Lin\'{e}aire, 31(3) (2014) 555--565.

\bibitem{cw}
D. Chae, J. Wolf, On partial regularity for the 3D nonstationary Hall magnetohydrodynamics equations on the plane, SIAM J. Math. Anal., 48(1) (2016), 443--469.

\bibitem{cl}
A. Cheskidov, X. Luo, Sharp nonuniqueness for the Navier--Stokes equations, Invent. Math., 229(3) (2022) 987--1054.

\bibitem{cl1}
A. Cheskidov, X. Luo, $L^2$-critical nonuniqueness for the 2D Navier--Stokes equations, Ann. PDE., 9 (2023) 13.

\bibitem{CL}
A. Cheskidov, X. Luo, Stationary and discontinuous weak solutions of the Navier--Stokes equations, arXiv:1901.07485v5.

\bibitem{CS}
D. Chae, M. Schonbek, On the temporal decay for the Hall-magnetohydrodynamic equations, J. Differ. Equ., 255 (2013) 3971--3982.

\bibitem{CWW}
D. Chae, R. Wan, J. Wu, Local well-posedness for the Hall-MHD equations with fractional magnetic diffusion, J. Math. Fluid Mech., 17 (2015) 627--638.

\bibitem{CW1}
D. Chae, S. Weng, Singularity formation for the incompressible Hall-MHD equations without resistivity, Ann. Inst. H. Poincar\'{e} Anal. Non Lin\'{e}aire, 33 (2016) 1009--1022.

\bibitem{CW}
D. Chae, J. Wolf, On partial regularity for the 3D non-stationary Hall magnetohydrodynamics equations on the plane, Comm. Math. Phys., 354 (2017) 213--230.

\bibitem{cdd}
M. Colombo, C. De Lellis, L. De Rosa, Ill-posedness of Leray solutions for the hypodissipative Navier--Stokes equations, Comm. Math. Phys., 362 (2018) 659--688.

\bibitem{cds}
S. Conti, C. De Lellis, and L. Sz\'{e}kelyhidi Jr., h-principle and rigidity for ${\rm c^{1,\alpha}}$ isometric embeddings, In Nonlinear partial differential equations, pages 83--116, Springer, 2012.

\bibitem{DM}
M. Dai, Regularity criterion for the 3D Hall-magneto-hydrodynamics, J. Differ. Equ., 261 (2016) 573--591.

\bibitem{DM1}
M. Dai, Nonunique Weak Solutions in Leray--Hopf Class for the Three-Dimensional Hall-MHD System, SIAM J. Math. Anal., 53(5) (2021) 5979--6016.

\bibitem{df}
M. Dai, S. Friedlander, Uniqueness and Non-Uniqueness Results for Forced Dyadic MHD Models, J. Nonlinear Sci., 33(1) (2023) 10.

\bibitem{DH1}
M. Dai, H. Liu, Long time behavior of solutions to the 3D Hall-magneto-hydrodynamics system with one diffusion, J. Differ. Equ., 266 (2019) 7658--7677.

\bibitem{DH}
M. Dai, H. Liu, On well-posedness of generalized Hall-magneto-hydrodynamics, Z. Angew. Math. Phys., 73 (2022) 139.

\bibitem{dt}
R. Danchin, J. Tan, On the well-posedness of the Hall-magnetohydrodynamics system in critical spaces, Commun. Partial Differ. Equ., 46(1) (2021) 31--65.

\bibitem{dt1}
R. Danchin, J. Tan, The global solvability of the Hall-magnetohydrodynamics system in critical Sobolev spaces, Commun. Contemp. Math., 24 (10) (2022) 2150099.

\bibitem{dl}
L. Du, X. Li, Sharp non-uniqueness for the 2D hyper-dissipative Navier-Stokes equations, arXiv:2407.06880.

\bibitem{dsdccm}
S. Donato, S. Servidio, P. Dmitruk, V. Carbone, M.A. Shay, P.A. Cassak, W.H. Matthaeus, Reconnection events in two-dimensional Hall magnetohydrodynamic turbulence, Phys. Plasmas, 19 (2012) 092307.

\bibitem{d}
L. De Rosa, Infinitely many Leray--Hopf solutions for the fractional Navier--Stokes equations, Commun. Partial Differ. Equ., 44(4) (2019) 335--365.

\bibitem{fl}
D. Faraco, S. Lindberg, Proof of Taylor's conjecture on magnetic helicity conservation, Comm. Math. Phys., 373(2) (2020) 707--738.

\bibitem{fls}
D. Faraco, S. Lindberg, L. Sz\'{e}kelyhidi Jr., Bounded solutions of ideal MHD with compact support in space-time, Arch. Ration. Mech. Anal., 239(1) (2021) 51--93.

\bibitem{f}
T.G. Forbes, Magnetic reconnection in solar flares,  Geophys. Astrophys. Fluid Dyn., 62 (1991) 15--36.

\bibitem{hg}
H. Homann, R. Grauer, Bifurcation analysis of magnetic reconnection in Hall-MHD systems, Phys. D, 208 (2005) 59--72.

\bibitem{h}
E. Hopf, ${\rm \ddot{U}}$ber die Anfangswertaufgabe f${\rm \ddot{u}}$r die hydrodynamischen Grundgleichungen, Math. Nachr, 4 (1951) 213--231.

\bibitem{kdb}
L. Kang, X. Deng, Q. Bie, Energy conservation for the nonhomogeneous incompressible ideal Hall-MHD equations, J. Math. Phys., 2021 (3) 62.

\bibitem{kryz}
A. Kiselev, L. Ryzhik, Y. Yao, A. Zlato\v{s}, Finite time singularity for the modified SQG patch equation, Ann. Math.,(2016) 909--948.

\bibitem{lhf}
F. Laakmann, K. Hu, P.E. Farrell, Structure-preserving and helicity-conserving finite element approximations and preconditioning for the Hall MHD equations, J. Comput. Phys., 492 (2023): 112410.

\bibitem{lt}
X. Li, Z. Tan, Non-uniqueness of weak solutions to 2D generalized Navier-Stokes equations, arXiv:2405.20754.

\bibitem{lqzz}
Y. Li, P. Qu, Z. Zeng, D. Zhang, Sharp non-uniqueness for the 3D hyperdissipative Navier--Stokes equations: above the Lions exponent, J. Math. Pure. Appl., 190 (2024) 103602.

\bibitem{lzz}
Y. Li, Z. Zeng, D. Zhang, Non-uniqueness of weak solutions to 3D magnetohydrodynamic equations, J. Math. Pure. Appl., 165 (2022) 232--285.

\bibitem{lzz1}
Y. Li, Z. Zeng, D. Zhang, Sharp non-uniqueness of weak solutions to 3D magnetohydrodynamic equations, J. Funct. Anal., 287(6) (2024) 110528.

\bibitem{leray}
J. Leray, Sur le mouvement d¡¯un liquide visqueux emplissant l¡¯espace, Acta Math., 63(1) (1934) 193--248.

\bibitem{l}
J.-L. Lions, Quelques m\'ethodes de r\'esolution des probl\`{e}mes aux limites non lin\'eaires, Dunod; Gauthier-Villars, Paris, 1969.

\bibitem{Luo}
X. Luo, Stationary solutions and nonuniqueness of weak solutions for the Navier--Stokes equations in high dimensions, Arch. Ration. Mech. Anal., 233(2) (2019) 701--747.

\bibitem{lq}
T. Luo, P. Qu, Non-uniqueness of weak solutions to 2D hypoviscous Navier--Stokes equations, J. Differ. Equ., 269(4) (2020) 2896--2919.

\bibitem{lt1}
T. Luo, E.S. Titi, Non-uniqueness of weak solutions to hyperviscous Navier--Stokes equations: on sharpness of J.-L. Lions exponent, Calc. Var. Partial Differ. Equ., 59(3) (2020) 92.

\bibitem{mb}
A.J. Majda, A.L. Bertozzi, Vorticity and Incompressible Flow, Cambridge Texts Appl. Math., Cambridge, 2002.

\bibitem{mgm}
P.D. Mininni, D.O. G\`{o}mez, S.M. Mahajan, Dynamo action in magnetohydrodynamics and Hall magnetohydrodynamics, Astrophys. J., 587 (2003) 472--481.

\bibitem{mny}
C. Miao, Y. Nie, W. Ye, On Onsager-type conjecture for the Els\"{a} sser energies of the ideal MHD equations, arXiv:2504.06071.

\bibitem{my}
C. Miao, W. Ye, On the weak solutions for the MHD systems with controllable total energy and cross helicity,J. Math. Pure. Appl., 181 (2024) 190--227.

\bibitem{ny}
Y. Nie, W. Ye, Sharp and strong non-uniqueness for the magneto-hydrodynamic equations, arXiv:2208.00228.

\bibitem{pw}
Y. Peng, H. Wang, Weak solutions to the Hall-MHD equations whose singular sets in time have Hausdorff dimension strictly less than 1, arXiv:2307.06587.

\bibitem{st}
M. Sermange, R. Temam, Some mathematical questions related to the MHD equations, Commun. Pure Appl. Math., 36 (1983) 635--664.

\bibitem{t}
T. Tao, Global regularity for a logarithmically supercritical hyperdissipative Navier--Stokes equation, Anal. PDE, 2(3) (2009) 361--366.

\bibitem{temam}
R. Temam, Navier--Stokes Equations. Theory and Numerical Analysis, North Holland, Amsterdam, 1977.

\bibitem{wyy}
Y. Wang, J. Yang, Y. Ye, On two conserved quantities in the inviscid electron and Hall magnetohydrodynamic equations, Nonl. Anal., 2025 (250) 113668.

\bibitem{w}
M. Wardle, Star formation and the Hall effect, Astrophys. Space Sci., 292 (2004) 317--323.

\bibitem{w1}
J. Wu, Generalized MHD equations, J. Differ. Equ., 195(2) (2003) 284--312.

\bibitem{y}
K. Yamazaki, Markov selections for the magnetohydrodynamics and the Hall-magnetohydrodynamics systems, J. Nonlinear Sci., 29(4) (2019) 1761--1812.
\end{thebibliography}
\end{document}